\documentclass[11pt]{amsart}

\usepackage[sort&compress,round]{natbib}

\usepackage{amsmath}
\usepackage{amsthm}
\usepackage{amssymb}

\usepackage{subcaption}   
\usepackage{graphicx}
\setkeys{Gin}{width=\linewidth,totalheight=\textheight,keepaspectratio}
\graphicspath{{./images/}}

\usepackage{booktabs}

\usepackage{units}

\usepackage{fancyvrb}
\fvset{fontsize=\normalsize}

\usepackage[T1]{fontenc}
\usepackage[sc]{mathpazo}
\usepackage{multicol}
\usepackage{multirow}
\usepackage{lipsum}
\usepackage{wrapfig}
\usepackage{latexsym}
\usepackage{mathrsfs}
\usepackage{mathtools}
\usepackage{enumerate}
\usepackage{chemarrow}
\usepackage[shortlabels]{enumitem}
\setlist{topsep=0.5em, itemsep=0em}
\usepackage{tensor}

\usepackage[math]{cellspace}
\usepackage{hyperref}
\hypersetup{
	hidelinks
}
\usepackage{cleveref}
\usepackage{color,cancel}
\usepackage[normalem]{ulem}
\usepackage{verbatim}
\usepackage{tcolorbox}
\usepackage{soul}

\usepackage{todonotes}

\newtheorem{theorem}{Theorem}[section]
\newtheorem{proposition}[theorem]{Proposition}
\newtheorem{definition}[theorem]{Definition}
\newtheorem{corollary}[theorem]{Corollary}
\newtheorem{lemma}[theorem]{Lemma}

\newcounter{assume}
\newtheorem{assumption}[assume]{Assumption}

\theoremstyle{definition}

\theoremstyle{remark}

\newcommand{\with}{\, ; \,}
\newcommand{\given}{\, | \,}
\newcommand{\suchthat}{\, : \,}

\newcommand{\rbb}{\mathbb{R}}

\newcommand{\ep}{\epsilon}
\renewcommand{\d}{\mathrm{d}}

\newcommand{\be}{\begin{equation}}
\newcommand{\ee}{\end{equation}}
\newcommand{\bea}{\begin{eqnarray}}
\newcommand{\eea}{\end{eqnarray}}

\DeclareMathSymbol{\umu}{\mathalpha}{operators}{0}

\def\E{\mathbb{E}}

\def\R{\mathbb{R}}

\def\LOPER{\mathcal{L}}

\def\SCK{K}
\def\CSEC{\sigma}

\def\VSP{{\Omega_v}}

\def\rbb{\mathbb{R}}

\newcommand{\Prob}{\mathbb{P}}
\newcommand{\domain}{\Omega}
\newcommand{\domainv}{\mathcal{V}}
\newcommand{\source}{\mathcal{S}}
\newcommand{\gen}{\mathcal{L}}
\newcommand{\Ind}{\mathbb{1}}

\DeclareMathOperator*{\argmin}{arg\,min}

\def\MAP{F^\epsilon}

\def\mydate{\number\day\ {\ifcase\month \or January\or February\or
              March\or April\or May\or June\or July\or August\or
              September\or October\or November\or December\fi}
\number\year}

\numberwithin{equation}{section}

\title{Source identification via path-wise gradient estimation.} %

\author[R. B. Lehoucq]{Richard B. Lehoucq}
\address{Sandia National Laboratories, Albuquerque, NM}
\email{rblehou@sandia.gov}

\author[S. A. McKinley]{Scott A. McKinley}
\address{Department of Mathematics, Tulane University, New Orleans, LA 70118, USA.}
\email{scott.mckinley@tulane.edu}

\author[P. Plech\'a\v{c}]{Petr Plech\'a\v{c}}
\address{Department of Mathematical Sciences, University of Delaware, Newark, DE 19716, USA}
\email{plechac@udel.edu}

\date{\today}

\begin{document}

\begin{abstract}
In the context of PDE-constrained optimization theory, source identification problems traditionally entail particles emerging from an unknown source distribution inside a domain, moving according to a prescribed stochastic process, e.g.~Brownian motion, and then exiting through the boundary of a compact domain. Given information about the flux of particles through the boundary of the domain, the challenge is to infer as much as possible about the source. 

In the PDE setting, it is usually assumed that the flux can be observed without error and at all points on the boundary. Here we consider a different, more statistical presentation of the problem, in which the data has the form of discrete counts of particles arriving at a set of disjoint detectors whose union is a strict subset of the boundary. In keeping with the primacy of the stochastic processes in the generation of the model, we present a stochastic gradient descent algorithm in which exit rates and parameter sensitivities are computed by simulations of particle paths. We present examples for both It\^o diffusion and piecewise-deterministic Markov processes, noting that the form of the sensitivities depends only on the parameterization of the source distribution and is universal among a large class of Markov processes.
\end{abstract}

\maketitle

\section{Introduction}
Source identification for elliptic PDEs is a classical inverse problem that has been considered in the context of gravimetry \citep{isakov1990inverse,isakov2006inverse}, EEG readings \citep{elbadia2000inverse,baratchart2004recovery} and more recently in several other more general settings (see \citet{liu2017inverse,liu2020reconstruction} for example). The theme of these problems is that particles emerge from a source in the interior of a compact domain and evolve according to a prescribed stochastic process until they exit the system, either through the boundary of the domain or through degradation. The particle system is assumed to be in a steady state, and the question is whether the location, size, and/or magnitude of the source distribution can be inferred from data taken exclusively at the boundary.

The problem is typically postulated in terms of a deterministic boundary value problem (BVP), formulated as follows. 
Suppose that $\mathcal{A}$  is an elliptic  
partial differential operator and $\domain \subset \mathbb{R}^d$ is an open domain with sufficiently regular boundary $\partial \domain$. Then the \emph{forward problem} consists of begin given source function $s:\domain \to \mathbb{R}$ and a boundary value function $f:\partial \domain \to \mathbb{R}$, and then solving
\begin{equation} \label{eq:bvp}
\begin{aligned}
\mathcal{A} u &= -s, & x \in \domain; \\
u &= f, &x \in \partial \domain.
\end{aligned}
\end{equation}
to find the boundary flux $g \colon \partial \domain \to \rbb$. Assuming the coefficient of the Laplacian term in $\mathcal{A}$ is constant and normalized to one, the boundary flux has the form
\begin{equation}
    g(x) \coloneqq \nabla u(x) \cdot n(x)
\end{equation}
for all $x \in \partial \domain$, where $n(x)$ is the outward normal vector at $x$.

The \emph{inverse problem} consists of inferring the source function $s$ from a given data pair $(f,g)$.
Such a problem is not always well-posed \citep{isakov1990inverse,liu2017inverse}. 
A typical example of unidentifiability occurs when two proposed sources have the same center and the same ``volume'', then the boundary data can be identical \citep{elbadia2000inverse,el2011inverse}. To deal with ill-posedness, investigators commonly restrict the source function to be a member of a reduced class of functions: e.g., point sources and dipoles \citep{elbadia2000inverse}, 
circles \citep{shigeta2003numerical}, and star-shapes \citep{alves2017inverse,liu2017inverse} are recent examples. 
When explicit solutions for $u$ are available, this yields a specific set of inverse methods that are, however, difficult to generalize. For this reason, a host of computational methods have been introduced, each tuned to the specifics of the challenges that the authors have in mind. A recent comprehensive review of numerical approaches to the inverse source problem for the Helmholtz equation can be found in the work by \citet{liu2020reconstruction}.

\subsection{Inverse problems for stochastic fountains: particle counts and binned data}

In this work, we take a different point of view on the source identification problems, which is rooted in a stochastic processes perspective. In particular, if we think of the operator $\mathcal{A}$ as the adjoint of the generator for a continuous-time Markov process, then we can interpret the inverse problem with Dirichlet data 
$f = 0$ as attempting to learn about the location of a source distribution given information about particle exits. 

This shift in perspective also encourages a change in assumptions about what we consider to be the ``data'' used for the source identification. 
For example, in PDE inverse problems it is commonly assumed that the pair of functions $(f,g)$ is specified at every point on the boundary. This means the observer must be able to measure the boundary flux pointwise, which is considered reasonable in the context of gravimetry. However, from the particle perspective, a more natural observable is a collection of particle \emph{counts} as they exit through a set of distinct regions $\{D_j\}_{j=1}^J$ that we call \emph{detectors}. From a probability theory perspective, this amounts to observing outcomes of a continuous random variable (the exit distribution) while restricted to binned data counts.

We can formalize this in the simplest case as follows. We model the particle source as a Poisson process of births on the real line that has rate $\lambda > 0$ and initial locations that are independently and identically distributed (i.i.d.) according to a probability distribution $\phi(x \with \theta)$, where $\theta \in \rbb^d$ is an unknown 
location parameter.  Suppose that the birth times are enumerated 
$\{T_1, T_2, \ldots,\}$ and the particle 
trajectories are labeled 
$\{X_n(t)\}_{t \geq T_n}$ for $n \in \mathbb{N}$. Upon arrival the particles undergo Markov process movement 
that is governed by a common generator $\gen$ until they exit the domain at hitting times denoted 
$\{\tau_n\}_{n \in \mathbb{N}}$. Then, for any given observation window $[0,T]$, the data we consider in this work have the form
$$
N_j(T) = \sum_{n=1}^\infty \Ind_{[0,T]}\big(\tau_n\big) \Ind_{D_j}\big(X_n(\tau_n)\big), \quad j \in \{1, \ldots, J\}\,,
$$
which counts the number of particles that have exited through detector $D_j$ in the interval $[0,T]$.
\begin{figure}[t]
    \centering
    \includegraphics[width = 0.9\textwidth]{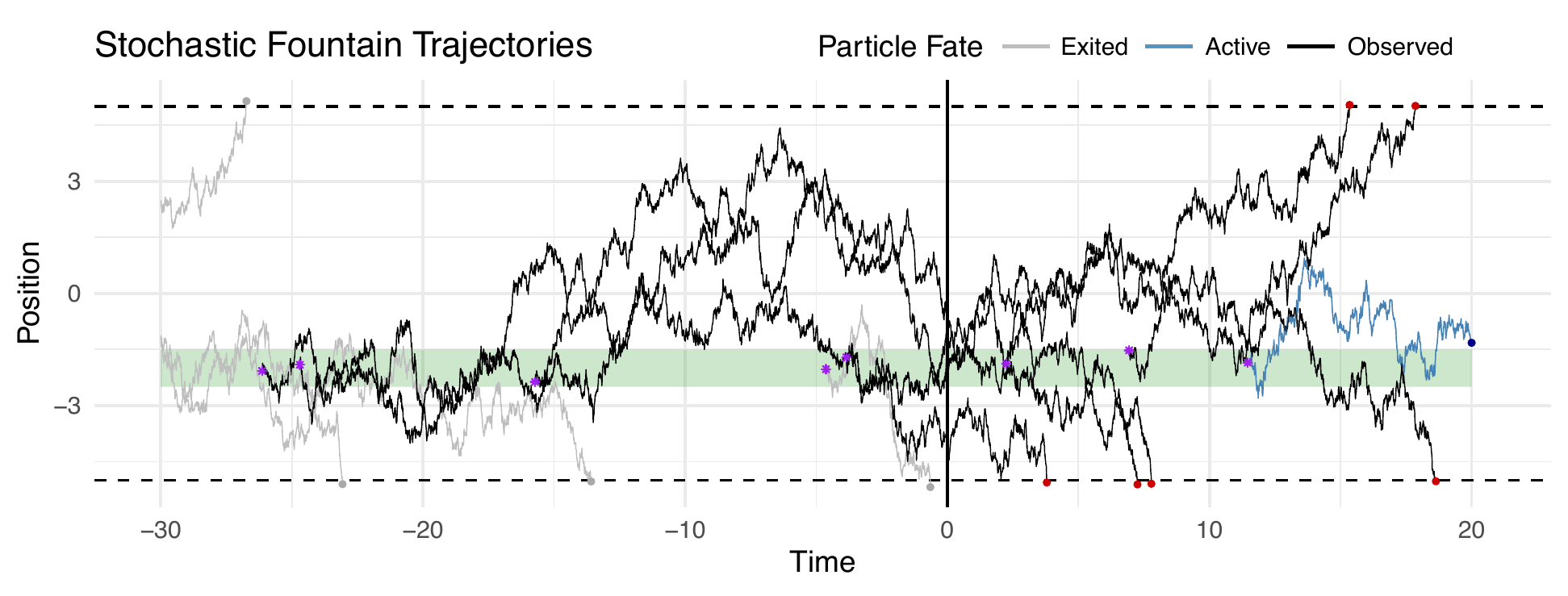} \\
    \includegraphics[width = 0.9\textwidth]{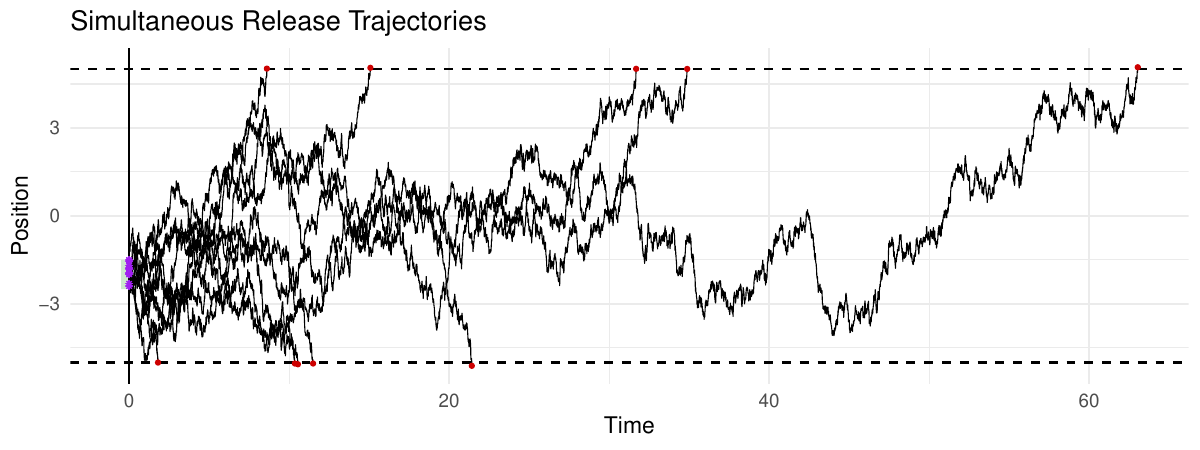}\\
    \caption{A \emph{stochastic fountain} is a particle ensemble model for experimentalists observing a system of dynamic particles ``in steady-state.'' Particles have birth times extending in the past ($t < 0$) and are observed over a prescribed window of time ($t \in [0,T]$). However, numerical simulations and stochastic process theorems are more commonly articulated in terms of all particles having the same initial time $t=0$, i.~e.~they are \emph{simultaneously released}, and simulated until exit. See \Cref{sec:fountain} for rigorous definitions.}
    \label{fig:fountain}
\end{figure}

A visualization of this \emph{stochastic fountain} system is provided in \Cref{fig:fountain}. In the top panel, we see particles emerging from the source distribution, whose support is indicated by the green band. Because the system is in steady-state, there are trajectories that emerge and fully resolve before time $t=0$ (gray paths). The paths that exit during the observation window (exit points marked in red and paths are colored black) may have been born before $t = 0$ or during the observation window $[0,T]$ (here $T = 20$). The fact that four particles in the example have exited through $x = -5$ while two exited through $x = 5$ provides a clue that the source might be off center. Clearly, more exit observations will be necessary to establish such a claim more rigorously. At the end of the observation window, some particles may remain active (blue paths) and inference aspects based on active particles at one snapshot in time has been explored by \citet{miles2024incorporating-a71}. 

We provide a contrast in the lower panel of \Cref{fig:fountain}, depicting what we refer to as a \emph{simultaneous release} system. While observations are naturally modeled by stochastic fountains, numerical simulations that follow paths all the way to their exit time are better articulated in terms of particles being simultaneously released at time $t=0$. 

\subsection{Connection between exit observations and BVP flux}

For the dynamics we have described, the particles are non-interacting, so there is a significant simplification in capturing the probability distribution of experimental outcomes. When the particle birth process is a Poisson process, the exits of particles through each of the detectors are thinned Poisson processes. In \Cref{sec:theory} we show that the numbers of particles observed to exit through the distinct detectors in the time interval $[0,T]$ are independent with respective distributions
\begin{displaymath}
    N_j \sim \text{Pois}\big(\lambda p_j (\theta) T\big)
\end{displaymath}
where $p_j(\theta)$ is probability that a given particle exits through the detector $D_j$ when its initial location is distributed according to $\phi(\cdot \with \theta)$. For reasons that will become clear later, we will commonly write the exit probabilities in terms of the mathematical expectation of a particle's location at its exit time $\tau$:
\begin{equation} \label{eq:expected-exit}
    p_j(\theta) \coloneqq \Prob^\theta\big(X(\tau)\in D_j\big)= \E^\theta\big[\Ind_{D_j}(X_\tau)\big]
\end{equation}
where $\Ind_A : \rbb^d \to \{0,1\}$ is the indicator function on a set $A \subset \rbb^d$.

To make the discussion concrete, suppose that particle locations evolve according to an It\^{o} diffusion with generator $\gen$. Define $\gen^*$ to be the adjoint of $\gen$ and let $u$ be the solution of the BVP
\begin{equation} \label{eq:bvp-u}
\begin{aligned}
    \gen^* u(x \with \theta) &= -\lambda \phi(x \with \theta), & x \in \domain; \\
    u(x \with \theta) &= 0, & x \in \partial \domain.
\end{aligned}
\end{equation}
From a PDE perspective, one would say the flux through the detectors is given by
\begin{equation} \label{eq:g_j}
g_j(\theta) \coloneqq \int_{D_j} \eta \nabla u(x \with \theta) \cdot n(x) \, \d S(x)
\end{equation}
where $D_j \subset \partial \domain$ and $\eta$ is the (constant) diffusion coefficient. One of our tasks in this work will be to confirm this flux interpretation matches with the stochastic process interpretation that particles are exiting the domain at the birth rate $\lambda$ times the appropriate exit probability, i.e.,
\begin{center}
    \emph{PDE flux through $D_j$ = Markov ensemble exit rate through $D_j$}.
\end{center}
This relationship exists in the stochastic processes ``folklore'' but to our knowledge has not been addressed in the context of statistical inference. As an example, a relationship of this type appears for It\^{o} diffusions in the text by \citet[Thm.~6.2.2]{schuss2009theory}, where the emphasis is placed on a single-point source with a unit birth rate. In this case, the boundary flux is exactly the exit distribution of any given particle \citep[Thm.~6.3.1]{schuss2009theory}. By contrast, in the work we present here, we wish to allow for more general arrival rates and to explore the exit-rate/BVP-flux in a more general Markov process context. To be specific, we will study stochastic fountains generated by (1) It\^{o} diffusions with additive noise; and (2), with an appropriately modified version of \Cref{eq:bvp-u} and \Cref{eq:g_j}, a class of piecewise-linear Markov processes, inspired both by neutron transport with scattering \citep{horton2023neutron} and intracellular cargo transport carried out by molecular motors \citep{cook2024considering}.

In terms of the respective inverse problems, the presence of discrete detectors turns the classical PDE inverse problem into one with binned data: i.~e.~we are given a vector of boundary fluxes 
$\hat{\mathbf{g}} = (\hat{g}_1, \ldots, \hat{g}_J)$ and use this to estimate the location parameter $\theta$ assuming that other parameters are known ($\lambda$ and $\eta$ in particular). One natural way to construct an estimator would be to define a convex loss function $\ell \colon \mathbb{R} \to \mathbb{R}_+$ and define
\begin{equation} \label{eq:optimization-pde}
\left\{
\begin{aligned}
    &\hat \theta_g \coloneqq \argmin_{\theta} \sum_j \ell\big(g_j(\theta) - \hat{g}_j\big), \\
    & \qquad \qquad\text{subject to } \eqref{eq:bvp-u} \text{ and } \eqref{eq:g_j}.
\end{aligned}
\right.
\end{equation}

From the stochastic processes perspective, the data takes the form of a vector of exit counts $\mathbf{\hat n} = (\hat n_1, \hat n_2, \ldots \hat n_J)$. Assuming that the particle arrival rate $\lambda$ and the duration of the observation window $T$ are known, our task is to build an estimator from the inferred exit probabilities $\mathbf{\hat p} = (\hat p_1, \hat p_2, \ldots, \hat p_J)$, which are constructed by normalizing the counts with respect to the expected number of total particle exits: 
\begin{equation} \label{eq:p-j-hat}
\hat p_j \coloneqq \frac{\hat n_j}{\lambda T}\,.
\end{equation} 
Analogous to \eqref{eq:optimization-pde}, we define our estimator for \emph{binned particle count data} as a constrained stochastic process optimization problem:
\begin{equation} \label{eq:optimization-process}
\left\{
\begin{aligned}
    &\hat{\theta} \coloneqq \argmin_{\theta} \sum_{j = 1}^J \ell \Big(\E^{\theta} \big[\Ind_{D_j}(X(\tau)) \big] - \hat{p}_j\Big) \\
    & \qquad \qquad\text{subject to } X(t) \text{ having generator } \gen, \\
    & \qquad \qquad\text{with initial location distribution } \phi(\cdot \with \theta).    
\end{aligned}
\right.
\end{equation}

Given these formulations, numerous search strategies can be postulated. In this work, we employ a gradient descent method, which presents a technical challenge that provided a major impetus for this research. 
In particular, taking the gradient of this objective function requires computing the sensitivity of the expected value with respect to the parameter of interest:
\begin{equation} \label{eq:gradient-general}
\nabla_\theta \left(\sum_{j = 1}^J \ell\big(\E^{\theta} \big[\Ind_{D_j}(X(\tau))\big] - \hat{p}_j)\right) = \sum_{j = 1}^J \ell'\big(\E^{\theta} \big[\Ind_{D_j}(X(\tau))\big]  - \hat{p}_j\big) \nabla_\theta \E^{\theta} \big[\Ind_{D_j}(X(\tau))\big].
\end{equation}
The essential observation of this work is that the form of the gradient term on the right-hand side does not depend on the particular choice of the Markov process, and both expectations on the right-hand side can be computed from the same ensemble of path-wise simulations.

\subsection{Approach to gradient estimation}

There is an extensive literature investigating sensitivities of expected values with respect to parameters, both for general random variables \citep{glynn1987likelilood} and for stochastic processes \citep{glynn1989likelihood}. The choice of method is usually determined by the overlap (or lack thereof) of the supports and relative magnitudes of the probability distributions used to generate the expected values. 

For example, suppose we would like to find the optimal choice from a candidate family of probability distributions, indexed by $\theta$, 
that have continuous densities. 
The support of each distribution is defined to be the set of all $x$ such that, for any open ball $B$ containing $x$, we have $\int_B f(x) \d x > 0$. 
When all distributions in the candidate family have the same support, we can use 
\emph{likelihood-ratio methods} that compute the Radon-Nikod\'ym derivative between candidate probability measures to assess the relative likelihood of experimental outcomes. 

In recent works, the parameter dependence of 
steady-state expectations computed from the stationary distribution of a Markov chain were studied in \citet{glynn2019likelihood}, for continuous time Markov chains (CTMC)
in \cite{wang_steady-state_2019}, and for diffusion processes in \citet{plechac_convergence_2021,plechac_martingale_2023}.
Sensitivity with respect to the initial conditions 
(the financial ``greek'', $\Delta$) was computed using likelihood ratio methods in \citet{broadie1996estimating}. 

However, if the initial condition of the generating stochastic process is just a single point, the family of candidate distribution functions is a set of Dirac $\delta$-functions 
respectively located at $\theta$. In this case, all candidate distributions are mutually singular, 
and the likelihood ratio methods are not well-defined. In this case, ideas from Malliavin calculus have proved useful, see \citep{fournie1999applications,chen2007malliavin,plyasunov2007efficient,warren2012malliavin,warren2013malliavin}.

Here, our model problem is similar to a common assumption made in the PDE inverse theory literature: namely, the source is supported on a ball centered at $\theta$ and therefore has a compact support. Since the Malliavin derivative technique is based on numerical integration of whole paths, 
it can become computationally intensive. 
Meanwhile, likelihood ratio approaches cannot be immediately implemented because any two candidate distributions under a perturbation of $\theta$
contain mutually singular subsets. However, our method can be seen as an extension of the likelihood ratio method, using an insight from the development of the classical Reynolds transport formula. We note that a similar ``shape derivative'' technique was implemented by \citet{liu2017inverse} and \citet{charkaoui2021efficient} for source identification techniques in the context of the Helmholtz equation with a variety of source support shapes.

\subsection{Pathwise Monte Carlo (MC) approach to estimation}

As has been widely noted in the literature,  differentiation of an expected value with respect to a parameter yields a Radon-Nikodym derivative formula that can be represented by an expectation of a random variable with respect to the base distribution. Roughly speaking, this means that for the gradient appearing in \Cref{eq:gradient-general}, there exists a random variable $Q$ such that 
\begin{equation}
    \nabla_\theta \E^\theta\big[\Ind_{D_j}(X(\tau))\big] \approx \E^\theta\big[\Ind_{D_j}(X(\tau)) Q \big].
\end{equation}
In \Cref{sec:fountain:sensitivity} we will find assumptions so that this relation is exact, and also find the corrective terms in occasions when it is not. Given such a formula though, if we can compute $Q$ from the same information that generates knowledge of the event $\{X(\tau) \in D_j\}$, then all terms in the gradient calculation \Cref{eq:gradient-general} can be estimated by MC methods and the MC path samples can be reused.

In fact, the use of MC methods is well-established in the study of radiation transport and can present several advantages over deterministic methods in specific geometries \citep{pascucci20042d}, when the medium in anisotropic \citep{pinte2009benchmark}, or when the data of interest involves particle counts or statistics concerning individual particle exit locations or exit directions \citep{code1995polarization}. See also \citet{brunner2002forms}, \citet{juvela2005efficient} and \citet{kuiper2013reliability} for interesting discussions concerning comparisons of computational methods. Building in a capability to simultaneously compute expected path outcomes and parameter sensitivities adds significant value to the methods.

Therefore, while the problem we present can be solved by deterministic methods, the framework generalizes to situations where MC methods have already been demonstrated to have intrinsic advantages. 

\subsection{Overview of work}

In \Cref{sec:fountain}, we define two classes of ``baseline'' Markov processes that can be readily simulated and for which source identification applications are natural to consider. We define what we call a steady state \emph{stochastic fountain} ensemble in which particles are born according to a Poisson arrival process and evolve independently until they exit a given compact domain. As in \Cref{fig:fountain}, we draw a contrast to simultaneous release systems that are implicitly generated in numerical experiments. 

In \Cref{sec:theory}, we provide a theoretical framework that explicitly ties the inverse problem articulated for stochastic fountains to the classical PDE source identification problem, as articulated by \Cref{eq:bvp-u,eq:g_j}, with appropriate modifications for piecewise-linear Markov processes (PLMPs). In \Cref{sec:fountain:sensitivity} we report one of our main theoretical results, implementing the formula to compute sensitivity of exit probabilities, and emphasize that the form of the result does not depend on the underlying baseline Markov process. 

Using the stochastic gradient descent algorithm articulated in \Cref{defn:descent}, in \Cref{sec:results}, we demonstrate the efficacy of the algorithm in a simple setting for both Brownian motion and PLMPs. We numerically explore two notions of statistical consistency. In \Cref{sec:results:M} we show convergence of the MC estimates for the gradient of the exit probability with respect to the source location parameter. In \Cref{sec:results:N}, we demonstrate a numerical version of consistency by examining the estimator for $\theta$ as more data points become available and near-perfect computation of the exit probabilities is achieved.

\section{Stochastic model development and notation}
\label{sec:fountain}

In this work, we consider two types of Markov processes: It\^{o} diffusions with constant diffusivity, and piecewise-linear Markov processes (PLMP) with constant-velocity segments. We define \emph{baseline} versions that are then replicated in two types of ensemble systems. The first, which we call a steady-state \emph{stochastic fountain} system, is the type of system we assume generates the data measured by our fictional observer. The second, a \emph{simultaneous release} system, is useful for the computational approximation of exit probabilities and derivatives with respect to parameters for the algorithm we report at the end of this section.

Our results apply to Markov processes that always include a location process $\{X(t)\}_{t \geq T_0}$, where $T_0 \in \rbb$ is some initial time. However, for PLMPs the location process is not Markov in and of itself. The velocity process must also be included for the system to satisfy the Markov property. Because applications and generalizations of this work may include even more state variables (for example, in the intracellular transport literature, the number and types of active molecular motors attached to the cargo are essential variables) we will use the notation $\{\chi(t)\}_{t \geq T_0}$ to refer to the full collection of state variables necessary for the process to be Markov.  In each case we will define the operator $\gen$ to be the generator of a continuous-time, time-homogeneous Markov process $\{\chi(t)\}_{t \geq T_0}$ in the sense that for all smooth and bounded functions $f$, we have 
\begin{displaymath}
    \gen f(x) = \lim_{t \downarrow T_0} \frac{1}{t} \big(\E^x[f(\chi_t)] - x\big)\,.
\end{displaymath}

\begin{assumption}[Baseline It\^{o} Diffusion]
\label{a:baseline-diffusion}
When $X(t)$ is an It\^{o} diffusion, we assume that it satisfies a stochastic differential equation (SDE) of the form
    \begin{equation}
    \begin{aligned}
        \d X(t) &= b(X(t)) \, \d t + \sqrt{2 \eta} \d W(t)\,, \;\quad t > T_0\,; \\
        X(T_0) &= \xi \sim \phi\,.
    \end{aligned}    
    \end{equation}
    where $b$ is a twice continuously differentiable vector field. In this case, the generator $\gen$ and its adjoint $\gen^*$ are
    \begin{equation}
    \begin{aligned}
        \gen f(x) &= b(x) \cdot \nabla f(x) + \eta \Delta f(x) \\
        \gen^* u(x) &= - \nabla \cdot \big(b(x) u(x)\big) + \eta \Delta u(x).
    \end{aligned}
    \end{equation}    
\end{assumption}

\begin{assumption}[Baseline Piecewise-Linear Markov Process (PLMP)]
\label{a:baseline-scattering}
A PLMP $\{\chi(t)\}_{t \geq T_0}$ can be expressed in terms of its joint location and velocity processes $\{X(t),V(t)\}_{t \geq T_0}$. We assume that the velocity process is a piecewise-constant function expressed in terms of successive velocities $v_i$ that are chosen according to a jump distribution $K(v,v',x)$ (informally speaking, the probability density for jumping from the velocity $v$ to the velocity $v'$ given that the particle is located at position $x$. If $\sigma_s(x,v)$ is a state-dependent velocity transition rate,
then the generator and its adjoint can be written as follows:
\begin{equation}\label{eq:PDMPgen}
\begin{aligned}
\LOPER f(x,v) &= v \cdot \nabla f(x,v) - \CSEC_s(x,v) f(x,v) + \CSEC_s(x,v)\int_{\VSP} \SCK(v,v',x)f(x,v')\,dv'\,,\\
\LOPER^* u(x,v) &= -v \cdot \nabla u(x,v) - \CSEC_s(x,v) u(x,v) + \CSEC_s(x,v)\int_{\VSP} \SCK(v',v,x)u(x,v')\,dv'\,.
\end{aligned}
\end{equation}
\end{assumption}

The dynamics of PLMPs are somewhat more transparent when expressed in terms of path properties. Suppose the initial position and velocity are chosen according to a joint distribution $(X_0,V_0) \sim \phi(x,v)$. The distribution of the first jump time $J_1$ is given by the survival probability
\begin{displaymath}
    \Prob\big(J_1 > t \given X_0 = x, V_0 = v\big) = e^{-\int_{T_0}^t \sigma_s(x + v t',v) \d t'}.
\end{displaymath}
Given the jump position, the new velocity has the distribution
\begin{displaymath}
    V_1 \big|_{(X(J_1), V_0) = (x,v)} \sim K(v, \cdot, x).
\end{displaymath}
Subsequent jump times and velocities are defined similarly, and the position process takes the form
\begin{equation}
    X(t) = X(J_{i}) + t V_i, \qquad t \in [J_i, J_{i+1})
\end{equation}
for $i \in \{0, 1, 2, \ldots\}$. 

An important instance of PLMPs arises in neutron transport, where the adjoint operator $\gen^*$ is the Boltzmann operator. See the recent text by \citet{horton2023neutron} for further information, and the paper by \citet{smith2025stochastic} for an SDE formulation of neutron transport. 

So far in this section, these stochastic processes have been defined in all $\rbb^d$. For the remainder of this work, we will be concerned with particles passing through the boundary of a compact domain. We next express our assumptions about the domain, describe admissible source distributions, and establish exit time notation for the analysis that follows.

\begin{assumption}[Domain] \label{a:domain}
    Let $\domain \subset \mathbb{R}^d$ be an open domain with smooth boundary $\partial \domain$. We will use $\overline \domain$ to denote the closure of the domain. Given a baseline Markov process $\{\chi(t)\}_{t \geq T_0}$ that has location $\{X(t)\}_{t \geq T_0}$, we define the exit time
    \begin{equation}
        \tau \coloneqq \inf \{t > T_0 \suchthat X(t) \notin \domain\}
    \end{equation}
    and assume that 
    \begin{equation}
        \sup_{x \in \domain} \E^x[\tau] < \infty.
    \end{equation}
\end{assumption}

\begin{definition}[Admissible source distributions]
\label{defn:source}
A function $\phi \colon \rbb^d \to \rbb_+$ is called an admissible source distribution if it is a probability distribution that is compactly supported on a ball of radius $\beta > 0$ centered at a location $\theta \in \rbb^d$ with radial symmetry about $\theta$. We will denote the vector of all source distribution parameters $\vartheta = (\theta, \beta)$. 

We say that a collection of source distributions form a \emph{location-and-scale} family if there is a continuously differentiable function $\Psi: [0,1] \to \rbb_+$ such that 
\begin{equation}\label{eq:spherical_density}
    \phi(x \with \vartheta) = \frac{1}{C \beta^d} \Psi\left(\frac{|x - \theta|}{\beta}\right)
\end{equation}
for all $\vartheta$ such that $B_\beta(\theta) \subset \domain$. Without loss of generality, we write  
\begin{equation} \label{eq:defn-base-distr}
\Psi(r) = \begin{cases}
    e^{-\psi(r)}, & r \in [0,1); \\
    \lim_{r' \uparrow 1} e^{-\psi(r')} & r = 1 \\
    0 & r > 1.
\end{cases}
\end{equation}
The constant $C$ depends only on $\psi$ and $d$ and satisfies
\begin{displaymath}
    C \coloneqq \frac{2 \pi^{d/2}}{\Gamma(d/2)} \int_0^1 e^{-\psi(r)} r^{d-1} \d r
\end{displaymath}
so that $\phi$ is a probability distribution.

We define 
\begin{equation} \label{eq:source-support}
    \partial\source_{\vartheta} = \{x \in \rbb^d \suchthat |x - \theta| = \beta\}
\end{equation} 
to be the boundary of the support of the source distribution.
\end{definition}

The two forms of the base distribution that we will keep in mind throughout this work have the following forms:
\begin{equation}
\begin{aligned}
    \text{Uniform distribution: } \psi(r) &= \Ind_{[0,1]}(r); \\
    \text{Bump function: } \psi(r) &= \frac{1}{1-r^2}.
\end{aligned} \hspace*{0.75 in}
\end{equation}

When the baseline Markov porcess is a PLMP, we will denote the space of all velocities $\domainv$ and extend the notation for $\phi$ to include velocities as an argument. Typically, we will assume that particles move with a fixed speed $c$ and the initial distribution of velocities will be uniform over a circle of radius $c$. 

\begin{definition}[Boundary Detectors and Exit Notation]
\label{defn:boundary-exit}
For a given domain $\Omega \in \rbb^d$, let $\{D_j\}_{j = 1}^J$ be a sequence of disjoint connected subsets of the boundary $\partial \Omega$. These represent detectors where we can observe particles exiting the domain. 

Suppose that a baseline Markov process $\{\chi(t)\}_{t \geq 0}$ has initial time $T_0 = 0$ and let $\{X(t)\}_{t \geq 0}$ be its location process. Let the domain $\Omega$, exit time $\tau$ and admissible source distribution $\phi$, satisfy the foregoing assumptions in this section. We define the survival function and the associated density of the exit distribution as follows:
\begin{equation} \label{eq:survival}
    S(t) = \Prob^\phi(\tau > t), \,\, \text{and } \,\, \rho(t) = - S'(t),
\end{equation}
where $\Prob^\phi$ indicates that the probability is conditioned on the initial distribution $\chi(0) \sim \phi$. For each $j$ we define the functions
\begin{equation} \label{eq:survival-restricted}
    S_j(t) = \Prob^\phi\big(\tau > t, X(\tau) \in D_j\big) \,\, \text{and } \, \rho_j(t) = - S_j'(t).
\end{equation}
That is to say, $S_j(t)$ is the proportion of particles that have not exited the domain as of time $t$ and will eventually exit through $D_j$. Because this is a restricted event rather than a conditioned event, $\rho_j(t)$ is not a probability density (because it does not integrate to one). For contrast, note that $\rho(t)$ \emph{is} a probability density. However, the integral of $\rho_j(t)$ is an essential quantity: it is the probability a particle exits through detector $j$, 
\begin{equation} \label{eq:surival-probability-detector}
\begin{aligned}
    p_j(\phi) \coloneqq \int_0^\infty \!\! \rho_j(t) \d t &= \lim_{t \to \infty} \int_0^t -S'_j(t') \d t' \\
    &= \lim_{t \to \infty} (S_j(0) - S_j(t)) = \Prob^\phi\big(X(\tau) \in D_j\big) 
\end{aligned}
\end{equation}
\end{definition}

We are now ready to rigorously define the \emph{stochastic fountain} system and the exit process that produces the data our theoretical observer will measure. The observer starts her detectors at time zero and counts the number of particles that exit through each detector until time $T > 0$. The \emph{fountain exit processes} refer to the arrivals at each after $t=0$ of the detectors as a function of time.

\begin{definition}[Steady-state stochastic fountain exit processes] \label{defn:fountain}
Let a generator $\gen$ be given for a baseline Markov process and let $\domain$ be an open domain with boundary detectors $\{D_j\}_{j=1}^J$. Let $\phi$ be an admissible source distribution. 

Let $A$ be a homogeneous Poisson process on the real line with intensity measure $\lambda$ times Lebesgue measure on $\rbb$. This represents the \emph{birth (or arrival) process} of particles in the fountain system. For a given sample of the point process $A(\omega)$ (here $\omega$ is an element of the underlying probability space), let $\{T_n(\omega)\}_{n=1}^\infty$ be an enumeration of its points.

For each $n \in \mathbb{N}$, let $\{X_n(t)\}_{t \geq T_n}$ be location of a Markov process driven by generator $\gen$ with initial condition $\chi_n(T_n) \stackrel{iid}{\sim} \phi$. For each $n$, let $\tau_n$ be the exit time for $X_n(t)$ from $\Omega$.

Then the exit process through detector $j$ is defined to be 
\begin{equation} \label{eq:fountain-exit}
    N_j(T) = \sum_{n=1}^\infty \Ind_{[0,T]}(\tau_n) \Ind_{D_j}\big(X_n(\tau_n)\big).
\end{equation}
\end{definition}

Our goal is to provide a method for using the fountain exit processes to infer the source location (and size, when possible). However, as mentioned in the introduction, we would like to take a pathwise approach to estimating exit probabilities. Since stochastic fountain systems are in steady state, numerical simulations of these systems entail producing significant burn-in periods. In practice, this means that a substantial number of paths are discarded when particles are born well before the observation window $[0,T]$ and exit the domain before time zero. (Recall the gray paths in the top panel of \Cref{fig:fountain}.)

In order to conduct a more efficient estimation of the exit probabilities, we introduce another particle system in which all trajectories are used directly in empirical averages. Here, a fixed number of particles $M$ are simulated starting at the initial time $t=0$ and simulated until their exit time, regardless of how long it takes. Through the theory presented in \Cref{sec:theory}, the empirical exit probabilities for the \emph{simultaneous release} system can be used to compute the exit probabilities and gradients that we need.

\begin{definition}[Simultaneous release exit probabilities]
\label{defn:simul-release}
Let a generator $\gen$ be given for a baseline Markov process and let $\domain$ be an open domain with boundary detectors $\{D_j\}_{j=1}^J$. Let $\phi$ be an admissible source distribution for particles within $\domain$. 

Let $M \in \mathbb{N}$ be a given number of particles and let $\{X_m(t)\}_{t \geq 0}$ be Markov processes driven by $\gen$ with initial condition $X_m(0) = \xi_m \stackrel{iid}{\sim} \phi$. For each $m$, let $\tau_m$ be the exit time for $X_m(t)$ from $\Omega$.

Then the empirical exit probability through detector $j$ is defined to be 
\begin{equation} \label{eq:fountain-exit-2}
    p_{j,M}(\phi) \coloneqq \frac{1}{M} \sum_{m=1}^M \Ind_{D_j}\big(X_m(\tau_m)\big).
\end{equation}
\end{definition}

\section{Analytical results: boundary flux and exit rates}
\label{sec:theory}
In this section we build our justification for using path simulations in a simultaneous release system (\Cref{defn:simul-release}) to estimate exit rates of the stochastic fountain system (\Cref{defn:fountain}) and ultimately the BVP flux presented in the introduction (\Cref{eq:bvp-u,eq:g_j}). The discussion that follows is restricted to It\^{o} diffusions with additive noise and PLMPs, but the formal argument should apply to any Markov process, as long as the adjoint operator $\gen^*$, integration by parts, and the boundary flux can be properly articulated.

\subsection{Results for It\^{o} diffusions}
\label{sec:theory:ito}

\begin{proposition}[Exit probabilities] 
\label{thm:dynkin}
Let $\domain$ be a domain satisfying the assumptions of \Cref{defn:source} with detectors specified by \Cref{defn:boundary-exit}. Let $\{\chi(t)\}_{t \geq 0}$ be a baseline Markov process satisfying \Cref{a:baseline-diffusion} with location process $\{X(t)\}_{t \geq 0}$ and admissible source distribution $\phi$ (\Cref{defn:source}). Then, for each $j \in 1, \ldots J$ the exit probabilities $p_j(\phi)$ satisfy
\begin{equation} \label{eq:exit-prob}
    p_j(\phi) \coloneqq \Prob^\phi(X(\tau) \in D_j) = \int_\Omega w_j(x) \phi(x) \d x,
\end{equation}
where $w_j(x)$ satisfies
\begin{equation} \label{eq:w-definition}
    \begin{aligned}
        \gen w_j(x) &= 0, & x \in \domain; \\
        w_j(x) &= \Ind_{D_j}(x), & x \in \partial \domain.
    \end{aligned}
\end{equation}
\end{proposition}
\begin{proof}
    This follows directly from Dynkin's formula. Suppose that $w_j(x)$ satisfies \Cref{eq:w-definition}. Then 
    \begin{equation} \label{eq:w-j-dynkin}
        \E^x[w_j(X(\tau))] = w_j(x) + \E^{x}\left[ \int_0^\tau \gen w_j(X(t)) \d t\right].
    \end{equation}
    By definition, $X(t) \in \domain$ for all $t < \tau$. Therefore, the integrand on the right-hand side is identically zero (since $\gen w_j(x) = 0$ in the interior of the domain). Meanwhile, $X(\tau) \in \partial \domain$, so we have on the left-hand side of \Cref{eq:w-j-dynkin} that $\E^x[w_j(X(\tau))] = \E^x[\Ind_{D_j}(X(\tau))] = \Prob^x(X(\tau) \in D_j)$. 
    
    We conclude that $w_j(x) = \Prob^x(X(\tau) \in D_j)$ and \Cref{eq:exit-prob} follows by integrating the initial condition over the source distribution.
\end{proof}

\begin{theorem}[Stochastic fountain exit rates]
\label{thm:fountain-exit-rates}
    Under the conditions given by \Cref{defn:fountain} with an It\^{o} diffusion baseline process satisfying \Cref{a:baseline-diffusion} and an admissible source distribution $\phi$ satisfying \Cref{defn:source}. Then the exit process $\{N_j(T)\}_{t \geq 0}$ defined by \Cref{eq:fountain-exit} is a Poisson process with rate $\lambda p_j(\phi)$ where $p_j(\phi)$ is the exit probability given in \Cref{thm:dynkin}.
\end{theorem}
\begin{proof}
    Our approach is to first establish that for all $t > 0$, $N_j(t)$ is a Poisson random variable. We then calculate its mean to establish the rate of the Poisson process.  As a preliminary step, we establish these facts for a system with an overall start time $t_0 < 0$ and the stated conclusions follow from taking the system start time $t_0$ back to $-\infty$. 
    
    To this end, let $\{\widetilde X_n(t)\}_{t \geq 0}$ ($n \in \mathbb{N}$) be a sequence of iid instances of particle location processes. These constitute the \emph{simultaneous release} ensemble described in \Cref{defn:simul-release} and displayed in the bottom panel of \Cref{fig:fountain}. Let $\{\tilde{\tau}_n\}$ denote their exit times from $\domain$. The number of particles in a system with start time $t_0 < 0$ and observation window $[0,T]$ is a random variable $N \sim \text{Pois}\big((T - t_0)\lambda\big)$. By standard Poisson process theory, we can assume that the individual particle birth times $\{T_n\}_{n = 1}^{N} \stackrel{iid}{\sim} \text{Unif}(t_0,T)$. For each $n$, define $X_n(t) := \widetilde{X}_n(t - T_n)$ for all $t \geq T_n$. Because the particles do not interact, their exits are independent, and the number that exit through each detector in the interval $[0,T]$is a thinned Poisson random variable. 

    Let $N_j(T \with t_0)$ denote the number of particles in a fountain system that exit through detector $j$ in the interval $[0,T]$ when the system start time is $t_0 < 0$. By Wald's equation, the identical distributions of the stochastic fountain particles, and the fact that the expectation of an indicator function is a probability of the indicated event, we have
    \begin{equation} 
        \begin{aligned}
            \E^\phi[N_j(t \with t_0)] &= \E^\phi\left[\sum_{n = 1}^N \Ind_{[0,T]}(\tau_n) \Ind_{D_j}(X\big(\tau_n)\big) \right]\\
            &= \E[N] \,\, \Prob^{\phi}\!\Big(\tau_1 \in [0,T], X_1(\tau_1) \in D_j\Big).
        \end{aligned}
    \end{equation}
    After observing that $E(N) = (T - t_0) \lambda$, we employ the law of total probability -- integrating over all possible fountain particle start times $T_1 \in [t_0,T]$:
    \begin{equation} \label{eq:fountain-intermediate}
    \begin{aligned}
        \E^\phi[N_j(t)] 
        &= \lambda (T - t_0) \,  \int_{t_0}^T \Prob^{\phi}\!\Big(\tau_1 \in [0, T], X_1(\tau_1) \in D_j \given T_1 = t \Big) \pi_{T_1}(t) \d t  \\
        &= \lambda \,  \int_{t_0}^T \Prob^{\phi}\!\Big(\tau \in [0, T], X_1(\tau_1) \in D_j \given T_1 = t \Big) \d t.
    \end{aligned}
    \end{equation}
    where $\pi_{T_1}(t) = 1/(T - t_0)$ is the (uniform) probability density of the particle birth time.
    
    Recalling the relationship $\widetilde{X}_1(t) = X_1(t - T_1)$, we have the following observations about the exit times in the simultaneous release time frame: If $T_1 < 0$, then the event $\{\tau_1 \in [0,T]\}$ equals the event $\{\tilde{\tau} \in [0-T_1, T-T_1]\}$. If $T_1 \geq 0$, then $\{\tau_1 \in [0,T]\} = \{\tau_1 \in [T_1,T]\} = \{\tilde{\tau}_1 \in [T_1,T]\}$. We can therefore proceed from \Cref{eq:fountain-intermediate} as follows:
    \begin{equation} \label{eq:fountain-almost-there}
    \begin{aligned}
        \E^\phi[N_j(T \with t_0)] 
        &= \lambda \, \Big(\int_{t_0}^T \Prob^{\phi}\!\Big(\widetilde{\tau}_1 \in [\max(0,0-t), T-t], \widetilde{X}_1(\widetilde{\tau}_1) \in D_j\Big) \d t  \\
        &= \lambda \int_{t_0}^{T} \Big(S_j\big(\max(0,-t)\big) - S_j(T - t)\Big) \d t.
    \end{aligned}
    \end{equation}
    To attain the expected number of exits through detector when the fountain is in steady state, we need to pull the system start time back to negative infinity. This is permitted because, by hypothesis, $\E^\phi(\widetilde{\tau}_1) < \infty$, and so 
    \begin{displaymath}
    \int_0^\infty S_j(y) \d y = \E^\phi\big[\widetilde{\tau}_1 \, \Ind_{D_j}\big(\widetilde{X}_1(\widetilde{\tau}_1)\big)\big] \leq \E^\phi(\tau) < \infty    
    \end{displaymath}
    It follows that
    \begin{displaymath}
        \E^\phi[N_j(T)] = \lambda \int_{-\infty}^T \Big(S_j\big(\max(0,-t)\big) - S_j(T - t)\Big) \d t.
    \end{displaymath}
    It remains to determine the exit rate, which can be attained by differentiating the above with respect to $T$. Indeed 
    \begin{equation}
    \begin{aligned}
        \frac{\d}{\d t} E^\phi[N_j(T)] &= \lambda \Big(S_j(\max(0,-T) - S_j(T) - \int_{-\infty}^T S_j'(T-t) \d t\Big) \\
        &= \lambda \int_{-\infty}^T \rho_j(T-t) \d t = \lambda p_j(\phi),
    \end{aligned}
    \end{equation}
    where we have invoked \Cref{eq:surival-probability-detector} from \Cref{defn:boundary-exit}.
\end{proof}

The previous results relate the exit rates of the steady-state fountain process to the exit probabilities captured by the simultaneous release system. It remains to relate these quantities to the boundary flux in the PDE inverse problems introduced at the beginning of this work. That result follows basically from integration by parts, relating the exit probabilities of the simultaneous release problem to the boundary flux of the BVP \Cref{eq:bvp-u,eq:g_j}.

\begin{theorem}[Relationship to BVP boundary flux]
    Define $u$ to be the solution to the BVP \eqref{eq:bvp-u}:
    \begin{displaymath}
    \begin{aligned}
        \gen^* u(x) &= -\lambda \phi(x), & x \in \domain; \\
        u(x) &= 0, & x \in \partial \domain.
    \end{aligned}
    \end{displaymath}
    where $\gen^*$ is the adjoint operator for a baseline It\^{o} diffusion process (\Cref{a:baseline-diffusion}). Then the boundary exit rates are related to the steady state PDE flux through the equation
    \begin{equation} \label{eq:diffusion-exit-flux}
        \lambda p_j(\phi) = \int_{D_j} \eta \nabla u(x) \cdot n(x) \, \d S(x) 
    \end{equation}
    where $n(x)$ is the outward pointing normal vector at the location $x$.
\end{theorem}
\begin{proof}
    From \Cref{a:baseline-diffusion}, $\gen^* u(x) =  - \nabla \cdot \big(b(x) u(x)\big)$ where $b(x)$ is a vector-valued function. We therefore have the product rule 
    \begin{displaymath} \nabla \cdot \big(w(x) b(x) u(x) \big) = (\nabla \cdot w(x)) b(x) u(x) + w(x) \nabla \cdot \big(b(x) u(x)\big).
    \end{displaymath}    
    Using this identity in conjunction with Green's second identity and the divergence theorem, we have
    \begin{displaymath}
    \begin{aligned}
        \langle w_j, \gen^* u \rangle_\domain &= \langle w_j , \eta \Delta u \rangle_\domain - \langle w_j , \nabla \cdot (a u) \rangle_\domain \\ 
        &= \langle \eta \Delta w_j , u \rangle_\domain - \eta \int_{\partial \domain} w_j(x) \big(\nabla u(x) \cdot n(x)\big) - \big(\nabla w_j(x) \cdot n(x)\big) u(x) \d S(x) \\
        & \qquad - \int_\domain \nabla \cdot \big(b(x) w_j(x) u(x)\big) \d x + \langle b \cdot \nabla w_j, u \rangle_\domain \\
        &= \langle \gen w_j, u \rangle + \int_{\partial \domain} \eta w_j(x) \big(\nabla u(x) \cdot n(x)\big) - \eta \big(\nabla w_j(x) \cdot n(x)\big) u(x) \\
        &\qquad \qquad \qquad \qquad \quad - \, w_j(x) u(x) (b(x) \cdot n(x)) \d S(x)
    \end{aligned}
    \end{displaymath}
    where $\langle f , g \rangle_\domain \coloneqq \int_\Omega f(x) g(x) \d x$. Noting that $\gen w_j = 0$ in the interior of the domain and that $u = 0$ and $w_j = \Ind_{D_j}$ on the boundary, we conclude that 
    \begin{displaymath}
        \langle w_j, \gen^* u\rangle = - \int_{D_j} \eta \nabla u(x) \cdot n(x) \d S(x)
    \end{displaymath}
    Now, using \Cref{thm:dynkin}, we have
    \begin{displaymath}
    \lambda p_j(\phi) = \langle w_j, \lambda \phi \rangle_\domain 
    = \int_{D_j} \eta \nabla u(x) \cdot n(x) \d S(x)
    \end{displaymath}
    which is what we intended to prove.
\end{proof}

\subsection{Results for Piecewise-Linear Markov Processes}
\label{sec:theory:plmp}

For PLMPs the associated PDEs are hyperbolic, so boundary conditions are only defined where characteristics are flowing into the domain. We follow the notation used by \citet{horton2023neutron}. Let $\domain$ satisfy \Cref{a:domain} and let $\domainv$ denote the velocity domain. Then we define
\begin{equation}
    \partial (\domain \times \domainv)^+ \coloneqq \{(x,v) \in \domain \times \domainv \suchthat x \in \partial \Omega \text{ and } v \cdot n(x) > 0\}
\end{equation}
with $\partial \Omega^-$ defined analogously. 

\begin{proposition}[Exit probabilities] 
\label{thm:dynkin-plmp}
Let $\domain$ be a domain satisfying the assumptions of \Cref{defn:source} with detectors specified by \Cref{defn:boundary-exit}. Let $\{\chi(t)\}_{t \geq 0}$ be a PLMP satisfying \Cref{a:baseline-scattering} and admissible source distribution $\phi$ (\Cref{defn:source}). Then, for each $j \in 1, \ldots J,$ the exit probabilities $\{p_j(\phi)\}$ satisfy
\begin{equation} \label{eq:exit-prob-plmp}
    p_j(\phi) = \int_\domain \int_{\domainv} w_j(x,v) \phi(x,v) \d x \d v,
\end{equation}
where
\begin{equation}\label{eq:w-pde-plmp}
    \begin{aligned}
        \gen w_j(x, v) &= 0, & (x,v) \in \domain \times \domainv; \\
        w_j(x,v) &= \Ind_{D_j}(x), & x \in \partial (\domain \times \domainv)^+.
    \end{aligned}
\end{equation}
\end{proposition}
\begin{proof}
    The proof directly follows that of \Cref{thm:dynkin}. Dynkin's formula holds in the same way, following the generalized It\^{o} formula for switching diffusions. See, for example, \citet[pg.~30]{yin2009hybrid} for more details. Suppose that $w_j(x)$ satisfies \Cref{eq:w-definition}. Then 
    \begin{displaymath}
        \E^{x,v}[w_j(X(\tau),V(\tau))] = w_j(x,v) + \E^{x,v}\left[\int_0^\tau \gen w_j(X(t), V(t)) \d t\right]\,.
    \end{displaymath}
    As before, the integrand on the right-hand side is identically zero and since $X(\tau) \in \partial (\domain \times \domainv)^+$, we have that
\begin{equation} \label{eq:w-conclusion-plmp}
    \begin{aligned}
    w_j(x,v) = \E^{x,v}[w_j(X(\tau),V(\tau))] = \Prob^{x,v}(X(\tau) \in D_j)\,. 
    \end{aligned}
\end{equation}  
\Cref{eq:exit-prob} follows by integrating the initial condition over the source distribution.
\end{proof}

\begin{theorem}[Stochastic fountain exit rates]
    Under the conditions given by \Cref{defn:fountain} with a PLMP satisfying \Cref{a:baseline-scattering}, with an admissible source distribution $\phi$ satisfying \Cref{defn:source}, the steady-state exit process $N_j(t)$ defined by \Cref{eq:fountain-exit} is a Poisson process with the rate $\lambda p_j(\phi)$ where $p_j(\phi)$ is given in \Cref{thm:dynkin-plmp}.
\end{theorem}
\begin{proof}
    Identical to \Cref{thm:fountain-exit-rates}.
\end{proof}

\begin{theorem}[Relationship to the BVP boundary flux]
    Define $u$ to be the solution to the BVP:
    \begin{equation}
    \label{eq:bvp-plmp}
    \begin{aligned}
        \gen^* u(x,v) &= -\lambda \phi(x, v), & x \in \domain \times \domainv; \hspace{0.75cm} \\
        u(x,v) &= 0, & x \in \partial (\domain \times \domainv)^-.
    \end{aligned}
    \end{equation}
    where $\gen^*$ is the adjoint operator for the baseline PLMP (\Cref{a:baseline-scattering}). Then the boundary exit rates are related to the steady-state PDE flux through the equation
    \begin{equation} \label{eq:plmp-exit-flux}
        \lambda p_j(\phi) = \int_{\partial^+} \!\! u(x) \Ind_{D_j}(x) (v \cdot n(x)) \, \d S(x) \d v.
    \end{equation}
\end{theorem}
The quantity on the right-hand side is the \emph{steady-state flux} through detector $j$ and, analogous to the diffusion context, we denote this quantity $g_j(\phi)$.
\begin{proof}
For a vector $v \in \rbb^d$ and scalar function $f : \rbb^d \to \rbb$, we note general relationship $\nabla \cdot \big(v f(x) \big) = v \cdot \nabla f(x)$. Together with the divergence theorem, for each fixed $v \in \domainv$, we have the following integration by parts formula (see, for example, \citet[pg.~45] 
{horton2023neutron})
\begin{displaymath}
\begin{aligned}
    &\int_{\domain \times \domainv} w(x,v) \big(v \cdot \nabla_x u(x,v)\big) \d x \d v \\
    & \qquad \qquad = \int_{\partial (\domain \times \domainv)} (v \cdot n(x)) w(x,v) u(x,v) \d x \d v - \int_{\domain \times \domainv} \big(v \cdot \nabla_x w(x,v)\big) u(x,v) \d x \d v
\end{aligned}    
\end{displaymath}
By definition, $u(x,v) = 0$ for $(x,v) \in \partial(\domain \times \domainv)^-$ and $w(x,v) = \Ind_{D_j}(x)$ for $(x,v) \in \partial(\domain \times \domainv)^+$. Recalling the definition of the PLMP flux \eqref{eq:plmp-exit-flux}, the above simplifies to 
\begin{displaymath}
\begin{aligned}
    &\int_{\domain \times \domainv} w(x,v) \big(v \cdot \nabla_x u(x,v)\big) \d x \d v \\
    & \qquad \qquad = g_j(\phi) - \int_{\domain \times \domainv} \big(v \cdot \nabla_x w(x,v)\big) u(x,v) \d x \d v
\end{aligned}    
\end{displaymath}

With this in mind, the exit probability for the $j$th detector satisfies
\begin{displaymath}
    \begin{aligned}
    \lambda \int_{\domain \times \domainv} w(x,v) \phi(x,v) \d x \d v &= - \int_{\domain \times \domainv} w(x,v) \gen^* u(x,v) \d x \d v \\
    &= \int_{\domainv}\int_{\Omega}  w_j(x,v)\, v\cdot \nabla_x u(x,v)\,\d x\, \d v \\ 
    & \qquad \qquad + \int_{\domainv} \int_{\Omega}  w_j(x,v) \big(\CSEC_s(x,v)\big) u(x,v)\d x\, \d v \\ 
    & \qquad \qquad - \int_{\domainv} \int_{\Omega}  w_j(x,v) \CSEC_s(x,v)\int_{\domainv} \SCK(v',v,x)u(x,v')\,\d v'\d x\, \d v \\ 
    &= g_j(\phi) - \int_{\domainv} \int_{\Omega} u(x,v) \, v \cdot \nabla_x w_j(x,v)\, \d x\, \d v \\ 
    & \qquad \qquad + \int_{\domainv} \int_{\Omega} u(x,v) \CSEC_s(x,v) w_j(x,v)\,\d x\, \d v \\ 
    & \qquad \qquad - \int_{\domainv} \int_{\Omega}  u(x,v) \CSEC_s(x,v)\int_{\domainv}\SCK(v,v',x)w_j(x,v')\,\d v'\d x\, \d v \\ 
    &= g_j(\theta) + \int_{\domainv} \int_{\Omega}   u(x,v) \, \gen w_j(x,v)\, \d x\, \d v.
    \end{aligned}
\end{displaymath}
where Fubini's theorem is used in the third equality. Since $\gen w_j(x,v) \equiv 0$ for $x \in \domain$, we have the desired result. This proof is a slight modification of Theorem 3.1 in \cite[pg.~45]{horton2023neutron}.
\end{proof}

\section{Analytical result: exit probability sensitivity}
\label{sec:fountain:sensitivity}

All of our results run through the observation that, for the fountain system we have described here, gradients can be computed from simulated data and restriction of the admissible source distributions to a compactly supported location-and-scale family. Importantly, the sensitivity with respect to this information has a universal form, meaning that no details about the baseline Markov generator $\gen$ appear in the statement of the sensitivity. The generator only appears through evaluations of the generated paths. We point out that due to a compactly supported distribution $\phi$ the usual log-likelihood sensitivity formula includes an additional term.

\begin{lemma} \label{thm:sensitivity}
Suppose that $\{\chi(t)\}_{t \geq 0}$ is a baseline  Markov process with location process $\{X(t)\}_{t \geq 0}$ in a domain $\domain$ satisfying \Cref{a:domain}. Suppose that the initial condition is distributed according to an admissible source distribution $\phi = \phi(\cdot \with \vartheta)$ with a compact support
$\mathcal{S}_\vartheta = \mathrm{supp}\,\phi(\cdot;\vartheta)\subset\R^d$.
Let $\tau$ be the exit time of the location process from the domain $\Omega$. Then for any smooth and bounded function $f : \domain \to \rbb$, the parametric sensitivity of $\E^\vartheta[f(X_\tau)]$ under the perturbation $\delta\vartheta$ 
is given by
\begin{equation}\label{eq:sensitivity} 
\begin{aligned}
\nabla_\vartheta \E^\vartheta[f(X_\tau)] \cdot \delta\vartheta &= \E^\vartheta\Big[ f(X_t) \,\zeta_\vartheta(X_0)\cdot\nabla_\vartheta\log\phi(X_0\with\vartheta)\Big] \\
& \qquad \qquad +
\int_{\partial \mathcal{S}_{\vartheta}} \E^x[f(X_t)] \phi(x\with\vartheta)\,\zeta_\vartheta(x)\cdot\nu_\mathcal{S}(x)\,dS\,,
\end{aligned}
\end{equation}
where $dS$ is the surface measure on the boundary of the initial condition support, $\partial \mathcal{S}_{\vartheta}$, and $\nu_\mathcal{S}$ is the outer normal on $\partial \mathcal{S}_{\vartheta}$.
The vector field $\zeta_\vartheta:\R^d\to\R^d$ is induced by the infinitesimal parameter perturbation $\delta\vartheta$ and the resulting change of the distribution support. 
\end{lemma}
\begin{proof}
In the proof we omit the subscript $\vartheta$ in the notation of the vector field $\zeta_\vartheta$. The deformation vector field is induced by the change of the support $\mathrm{supp}\,\phi(\cdot,\vartheta+\delta\vartheta)$.
We denote
\[
q(x) = \E^x[f(X_\tau)]\,,
\]
for a process with $X_0 = x$. When integrating against a source distribution characterized by the parameter vector $\vartheta$, we write
\[
\Phi(\vartheta) \coloneqq \E^\vartheta[f(X_\tau)]\equiv \int_{\mathcal{S}_\vartheta} q(x)\,\phi(x\with\vartheta)\,dx\,.
\]
The change of the parameter $\vartheta$ also induces a change of the support $\mathcal{S}_{\vartheta}$ to
\begin{displaymath}
\mathcal{S}_{\vartheta+\epsilon\delta\vartheta} = \{ x + \epsilon\zeta(x)\,|\,x\in\mathcal{S}_\vartheta\}\,.
\end{displaymath}
Assuming that $\zeta:\R^d\to\R^d$ is a Lipschitz vector field then for sufficiently small $\epsilon$ the mapping $\MAP(x) \coloneqq x + \epsilon\zeta(x)$ 
is a diffeomorphism, and we can use the change of variables formula
to evaluate the integral over the perturbed domain 
$\mathcal{S}_{\vartheta+\epsilon\delta\vartheta}$ to obtain 
\begin{equation}\label{eq:changevars}
\begin{aligned}
\Phi(\vartheta+\epsilon\delta\vartheta) & \equiv\int_{\mathcal{S}_{\vartheta+\epsilon\delta\vartheta}} q(x) \phi(x\with\vartheta+\epsilon\delta\vartheta) \,\text{d} x \\
&= 
\int_{\mathcal{S}_{\vartheta}} 
q(F^\ep(x))\, \phi(F^\ep(x)\with \vartheta + \epsilon \delta \vartheta))\,\, 
|\det\nabla\MAP|
\,\text{d}x\,. \nonumber
\end{aligned}
\end{equation}
Our assumption on integrability of $q\,\phi$ and its derivatives grants the expansions
\begin{displaymath}
\begin{aligned}
q(x+\epsilon\zeta)\,\phi(x+\epsilon\zeta\with\vartheta+\epsilon\delta\vartheta) & = 
q(x)\,\phi(x\with\vartheta) + 
\epsilon\nabla q(x) \cdot\zeta(x) \,\phi(x\with\vartheta)  \\
& \qquad + \epsilon q(x) \nabla_x\phi(x\with\vartheta) \cdot\zeta(x) \\
& \qquad + \epsilon q(x) \,\nabla_\vartheta \phi(x\with\vartheta) \cdot \delta\vartheta + o(\epsilon), \text{ and}\\
\det\nabla\MAP & = \det(\mathrm{Id} + \epsilon\nabla\zeta) = 
1 + \epsilon\mathrm{Tr}\nabla\zeta + o(\epsilon)\,. \hspace{2 cm}
\end{aligned}    
\end{displaymath}
We can then rewrite \eqref{eq:changevars} as
\[
\begin{split}
\int_{\mathcal{S}_{\vartheta+\epsilon\delta\vartheta}} q(x) &\phi(x\with\vartheta+\epsilon\delta\vartheta)\,\text{d}x = 
\int_{\mathcal{S}_{\vartheta}} q(x)\,\phi(x\with\vartheta) \,\text{d} x \\
 &+ \epsilon\int_{\mathcal{S}_{\vartheta}} \nabla_x q(x)\cdot\zeta(x) \,\phi(x\with\vartheta)\,\text{d}x 
 + \epsilon\int_{\mathcal{S}_{\vartheta}} q(x) \nabla_x\phi(x\with\vartheta) \cdot\zeta(x) \,\text{d}x \\
 &+ \epsilon\int_{\mathcal{S}_{\vartheta}} q(x) \,\nabla_\vartheta\phi(x\with\vartheta)\cdot\delta\vartheta \,\text{d}x \\
 &+  \epsilon \int_{\mathcal{S}_{\vartheta}} q(x)\phi(x\with\vartheta)\,\mathrm{div}\zeta(x) \,\text{d}x + o(\epsilon)\,.
\end{split}
\]
Thus using $\mathrm{div}(q\phi\zeta) = \nabla_x(q\phi)\cdot\zeta + q\phi\,\mathrm{div}\zeta$ and the integration by parts (Gauss theorem) we
have the directional sensitivity with respect to the 
initial distribution:
\begin{equation}
\lim_{\epsilon\to 0} \frac{\Phi(\vartheta+\epsilon\delta\vartheta) - \Phi(\vartheta)}{\epsilon} = 
\int_{\mathcal{S}_{\vartheta}} q(x) \,\nabla_\vartheta\phi(x\with\vartheta)\cdot\delta\vartheta \,\text{d}x  +
\int_{\partial \mathcal{S}_{\vartheta}} q(x)\phi(x\with\vartheta)\,\zeta\cdot\nu\,\text{d}S\,,
\end{equation}
establishing the path-wise form of the sensitivity formula \eqref{eq:sensitivity}.
\end{proof}

For the case considered here, where the source is supported on the unit ball of the radius $\beta$ centered at $\theta$ with the spherically symmetric source density given by \eqref{eq:spherical_density}, we have the vector fields corresponding to perturbations of $\vartheta=(\theta,\beta)$ 
\begin{displaymath}
    \begin{aligned}
    & \zeta^i_{\theta_j} = \delta_{ij}\,, & \mbox{for computing $\partial_{\theta_i}$, $i=1,\dots,d$; } \\
    & \zeta_{\beta} = \frac{\displaystyle x - \theta}{\displaystyle |x -\theta|} \,, & \mbox{for computing $\partial_\beta$.}        
    \end{aligned}
\end{displaymath}
Furthermore, using the density $\phi$ given by \eqref{eq:spherical_density} straightforward calculations yield
\[
\begin{aligned}
    \frac{\partial\phi}{\partial\theta_i}(x\with \theta,\beta) &= 
    \frac{1}{\beta}\phi(x\with \theta,\beta)\psi^\prime\left(\frac{|x-\theta|}{\beta}\right) \frac{x_i-\theta_i}{|x-\theta|}\,, \\
    \frac{\partial\phi}{\partial\beta}(x\with \theta,\beta) &=
    \frac{1}{\beta}\phi(x\with \theta,\beta)\psi^\prime\left(\frac{|x-\theta|}{\beta}\right) \left(\frac{|x-\theta|}{\beta} - d\right). \\
\end{aligned}
\]
Substituting \eqref{eq:sensitivity} together with the counting observable $f(x) = \Ind_{D_j}(x)$ (which can be seen as the limit of a sequence of smooth bounded functions for which \Cref{thm:sensitivity} applies), we obtain formulas for the two-dimensional case considered in the numerical simulations in Section~\ref{sec:results}.
%
%
\begin{corollary}
Suppose that $\{X(t)\}_{t \geq 0}$ is the two-dimensional location process of a baseline Markov process with initial condition distributed according to an admissible source distribution $\phi = \phi(\cdot \with \theta, \beta)$ in a domain $\domain$ satisfying \Cref{a:domain}. Then
\begin{equation}\label{eq:grad}
\begin{aligned}
    \nabla_\theta \E^{\phi}\big[\Ind_{D_i}(X_\tau)\big] &= \frac{1}{\beta} \E^{\phi}\left[\Ind_{D_j}(X_\tau) 
    \psi^\prime\left(\frac{|X_0 - \theta|}{\beta}\right)
    \frac{X_0 - \theta}{|X_0 - \theta|}\right] \\
    & \qquad \qquad + \frac{\Psi(1)}{C \beta^d} \int_0^{2 \pi} \E^{x(\alpha)}[\Ind_{D_j}(X_\tau)] \binom{\cos(\alpha)}{\sin(\alpha)} \d \alpha; \text{ and,}
    \\
    \partial_\beta \E^{\phi}\big[\Ind_{D_j}(X_\tau)\big] &= 
    \frac{1}{\beta}\E^{\phi}\left[\Ind_{D_j}(X_\tau) 
    \psi^\prime\left(\frac{|X_0 - \theta|}{\beta}\right)
    \left(\frac{X_0 - \theta}{\beta} - d \right)\right] \\
    & \qquad \qquad + \frac{\Psi(1)}{C \beta^d} \int_0^{2\pi} \E^{x(\alpha)}[\Ind_{D_j}(X_\tau)] \, \d \alpha;
\end{aligned}
\end{equation}
where $x(\alpha) = \binom{\theta + \beta \cos(\alpha)}{\theta + \beta \sin(\alpha)}$ parametrizes the boundary of the source support.
\end{corollary}

At this point the advantage of using a mollified indicator function (the bump function) for our source distribution becomes apparent. Because the limit of the source distribution from the interior of the support is what appears the boundary terms of \Cref{eq:grad}, use of an indicator function requires a sampling scheme on the surface of the source distribution and additional computation. On the other hand, the bump function approaches zero at its support boundary and therefore there is no contribution from a boundary term when computing gradients. The trade-off is the care that must be taken to sample from the bump function. We take an importance sampling approach, where elements are drawn from the uniform distribution on the support of the source, and then outcomes are re-weighted by the relative likelihood of that sampled initial condition coming from the bump function distribution.

\begin{definition}[Monte Carlo approximation scheme.]
\label{defn:MC-samples}
    Let $M \in \mathbb{N}$ and an admissible parameter vector $(\theta,\beta)$ be given. Let $U_m$ be an iid sequence of draws from the uniform distribution on the unit ball and for each $m$ define $\xi_m = \theta + \beta U_m$.
    
    Let $\{X(t)\}_{t \geq 0}$ be the location process of a baseline Markov process driven by the generator $\gen$ with initial location $X_0 \sim \phi(\cdot \with \theta, \beta)$ and $\domain$-exit time $\tau$. Let $\{X_m(t)\}$ be an iid sequence of simulated Markov processes driven by $\gen$ with initial conditions $X_m(0) = \xi_m$, and let $\{\tau_m\}$ be there associated $\domain$-exit times. 

    Then we define
\begin{equation} \label{eq:sensitivity-pathwise}
\begin{aligned}
    \widehat{E}^\theta\big[\Ind_{D_j}(X(\tau))\big] &= \frac{1}{M} \sum_{m = 1}^M \Ind_{D_i}\big(X_{m}(\tau)\big)
    \phi(U_m\with 0,1)\,\mathrm{Vol}(B_\beta)\,,  \\
    \widehat{\nabla_\theta \E^\theta}\big[\Ind_{D_j}(X(\tau))\big] &= 
    \frac{1}{M} \sum_{m = 1}^M \Ind_{D_j}\big(X_{m}(\tau)\big)
    \psi^\prime\left(\frac{|\xi_m - \theta|}{\beta}\right)\frac{\xi_m - \theta}{\beta |\xi_m - \theta|}  \phi(U_m \with 0,1) \mathrm{Vol}(B_\beta)  \,,
\end{aligned}
\end{equation}
\end{definition}

With the above definitions in hand, we can articulate the individual steps of our source location identification algorithm.

\begin{definition}[Gradient descent via path simulations] \label{defn:descent}
    Let $\{M_k\}$, $k \in \mathbb{N}$, denote a sequence of ensemble sizes to be used for exit probability and sensitivity estimation. Let $\{h_k\} \subset \mathbb{R}_+$, $k \in \mathbb{N}$, denote a sequence of algorithmic step sizes. Let $K$ either be a pre-defined number of algorithm steps, or a prescribed stopping rule.

    Then for a given data vector $\widehat{\mathbf{p}}$ and initial parameter guess $(\tilde \theta_0, \beta^0)$, where $\beta^0$ is the (assumed to be known) source distribution radius. We define the sequence of location parameter approximations $\{\tilde \theta_k\}_{k=1}^K$ by the iteration
    \begin{equation}
        \tilde \theta_{k + 1} \coloneqq \tilde \theta_k - h_k \sum_{i=1}^D \left(\widehat{p}_j - \widehat{E}^{\tilde \theta_k}\big(\Ind_{D_i}(X_\tau)\big)\right) \widehat{\nabla_{\theta} E^{\tilde \theta_k}}\big(\Ind_{D_i}(X_\tau)\big)
    \end{equation}
    where the estimated exit probability and gradient is defined by \Cref{defn:MC-samples}, computed with $M = M_k$ and $\vartheta = (\tilde \theta_k, \beta^0)$ respectively.
    
    Our algorithmic source location estimator is then
    \begin{equation}
        \widehat{\theta}_{M}(\widehat{\mathbf{p}}) \coloneqq \tilde \theta_K.
    \end{equation}
\end{definition}

\section{Numerical Results}
\label{sec:results}

We now share some simple demonstrations of the inference algorithm. In this section, we focus on estimating the source location and assume that the size of the detector is known. Due to the content in \Cref{sec:theory}, we have shown that exit rate estimation for the stochastic fountain system can be pursued using simulations of the simultaneous release system.
Indeed in \Cref{sec:results:MN} we provide estimation paths for multiple experimental settings: two It\^{o} diffusions and three PLMPs. 

In subsequent sections, we explore the impact of the two main sources of estimator variance. In \Cref{sec:results:M} we show the impact of $M$, which denotes the number of sample paths used in Monte Carlo approximation of exit probabilities and associated sensitivities with respect to initial condition. In \Cref{sec:results:N} we turn our attention to the statistical notion of consistency by numerically assessing the reduction in estimator variance that follows from increasing $N$, which denotes the number of particle exit observations that comprise ``experimental data.'' Ultimately we are able to show a contrast in the magnitudes of these two sources of error, and see that in our chosen scenarios, the size of the experimental data $N$ is a greater constraint than the size of the path simulation ensembles $M$.

\subsection{Pathwise estimates of gradients is effective across multiple generating Markov processses.} 
\label{sec:results:MN}

In our numerical experiments we took the domain $\domain$ to be the unit circle centered at the origin. Five detectors were equally spaced on the boundary. We set the source location to be $\theta^0 = (-0.4, 0.1)$ and scale of the source distribution was $\beta^0 = 0.15$. The parameters of the stochastic processes are listed in the captions of \Cref{fig:trajectories}. We constructed two scenarios for It\^o diffusion and two scenarios for the transport process. In the PLMP scenarios we add particle absorption as a possible outcome. Even though we excluded absorption in the foregoing theoretical development, the results hold with appropriate modification. To emphasize the robustness of the results with respect to adding particle absorption, in this section we introduce the possibility of a state-dependent absorption rate $\sigma_a$. Given a particle path, the absorption time $\tau_{\mathrm{abs}}$ is defined by the survival function
\begin{displaymath}
    \Prob_{x,v}\big(\tau_{\mathrm{abs}} > t\big) = e^{-\int_0^t \sigma_a(\chi(t')) \d t'}\,.
\end{displaymath}
If a particle's absorption time is less than its exit time, then it will not be counted in a detector's exit count.

The four numerical experiments were set up as follows:
\begin{displaymath}
    \begin{aligned}
        \text{Experiment 1},& \text{ Fig.~\ref{fig:trajectory_brownian}:} \quad \text{It\^{o} Diffusion, } b = (0,0), \eta = 1/2;\\
        \text{Experiment 2,}& \text{ Fig.~\ref{fig:trajectory_diffusion}:} \quad \text{It\^{o} Diffusion, } b = (-2,2), \eta = 1/2; \\
        \text{Experiment 3,}& \text{ Fig.~\ref{fig:trajectory_scattering_uniform}:} \quad \text{PLMP, } c = 0.1, \sigma_a = 0.1, \sigma_s = 0.8,\\ 
        &\qquad \qquad \qquad \qquad \pi_s \sim \text{Unif}(0,2\pi);\\
        \text{Experiment 4,}& \text{ Fig.~\ref{fig:trajectory_scattering_preferred}:} \quad \text{PLMP, } c = 0.1, \sigma_a = 0.1, \sigma_s = 0.8, \\
        & \qquad \qquad \qquad \quad \quad \pi_s \sim \mathrm{Norm}(\pi/3,\varsigma^2 \with (0,2\pi]),\\ 
        &\qquad \qquad \qquad \qquad \mbox{with two cases $\varsigma=2$, and $\varsigma=10$.}
    \end{aligned}
\end{displaymath}
For the PLMP cases, the velocity domain $\domainv$ is a circle of radius $c$. The stated distributions $\pi_s$ refer to the choice of angle off the $x$-axis for each velocity state, and induces the jump distribution $K(v, v',x)$ through the relatinship $v = \big(c \cos(\alpha), c \sin(\alpha)\big)$, where $\alpha \sim \pi_s$. (Note that our scattering distribution is does not depend on location or the current angle of motion.) The notation $\mathrm{Norm}(\pi/3,\varsigma^2 \with (0,2\pi])$ for the scattering distribution $\pi_s$ in the last experiment refers to the truncated Gaussian distribution with the mean $\pi/3$ and variance $\varsigma^2$.

For the diffusion cases, we sought to generate the exit probability data and associated inference using different techniques: for the data, we used the BVP characterization to compute exit probabilities; and for the inference scheme, we used path simulations for exit probabilities and gradient estimation.

\begin{figure}[h!]
    \centering
    \begin{subfigure}{0.45\textwidth}
        \includegraphics[width=\textwidth]{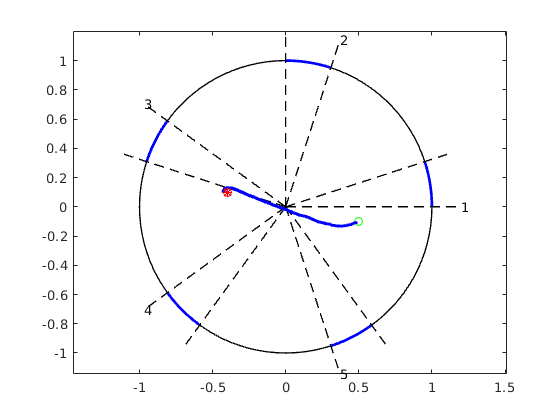}~
        \caption{Brownian motion, $N=\infty$, $M=10^4$.}
        \label{fig:trajectory_brownian}
    \end{subfigure}
    \hfill
    \begin{subfigure}{0.45\textwidth}
        \includegraphics[width=\textwidth]{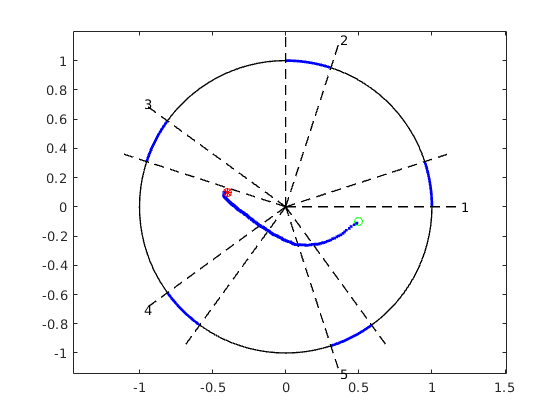}~
        \caption{Diffusion with the drift $b = \binom{-2}{2}$, $N=5\times 10^4$, $M=10^4$.}
        \label{fig:trajectory_diffusion}        
    \end{subfigure}

    \begin{subfigure}{0.45\textwidth}
        \includegraphics[width=\textwidth]{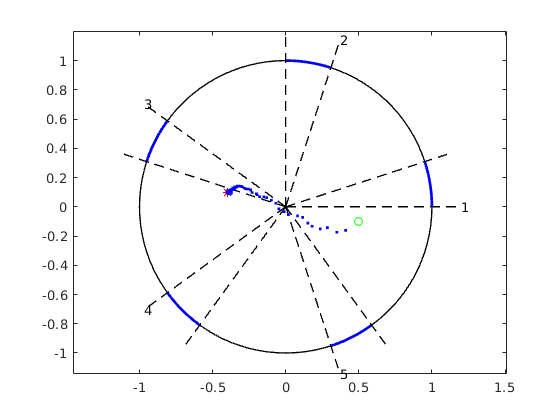}~
        \caption{Transport process with uniform scattering on $S^1$, and the rates $\sigma_a = 0.1$ (absorption), $\sigma_s = 0.8$ (scattering), $N=5\times 10^4$, 
        $M=10^4$.}
        \label{fig:trajectory_scattering_uniform} 
    \end{subfigure}
    \hfill
    \begin{subfigure}{0.45\textwidth}
        \includegraphics[width=\textwidth]{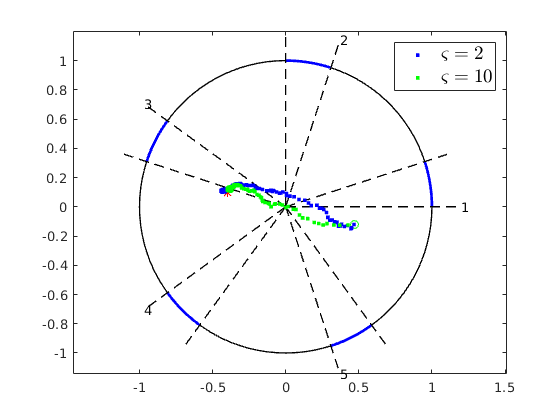}~
        \caption{Transport process with scattering on $S^1$, to a preferred direction; a random direction 
        from the truncated normal distribution with the mean $\pi/3$, $\varsigma=2$, $\varsigma=10$, and the rates $\sigma_a = 0.1$ (absorption), $\sigma_s = 0.8$ (scattering), $N=5\times 10^4$, $M=10^4$}
        \label{fig:trajectory_scattering_preferred} 
    \end{subfigure}
    \caption{Source identification trajectories for the four experimental scenarios described in \Cref{sec:results} using stochastic gradient descent (\Cref{defn:descent}) with $M = 10^4$, $K = 1000$ and step size $h=0.01$.}\label{fig:trajectories}
\end{figure}

To be precise, in the It\^{o} diffusion cases, we simulated exit counts in the following way. First, we articulated the exit probabilities as a function of initial condition in terms of the BVP \eqref{eq:w-definition}. We then numerically integrated against the source distribution centered at $\theta^0$ to establish the ``ground truth'' probability of exit through each detector by particles whose initial conditions were drawn from $\phi^0 = \phi(\cdot \with \theta^0, \beta^0)$:
\begin{equation} \label{eq:exit-prob-ground-truth}
    p_j(\phi^0) \coloneqq \langle w_j, \phi^0 \rangle_\domain
\end{equation}
To compute these quantities, we used importance sampling to approximate the integrals against the source distribution, \Cref{eq:exit-prob-ground-truth}. We did this by taking sample locations $\xi_i \sim \text{Unif}\big(\text{supp}(\phi^0)\big)$, evaluating $w_j(\xi_i)$, and weighting the exit probabilities by their relative likelihood to be drawn from the source distribution $\phi^0$. That is 
\begin{equation}\label{eq:p-bar-mc}
p_j(\phi) = \mathrm{Vol}(\mathcal{B}_\beta) \int_\Omega w_j(x) \phi(x)\,\frac{\d x}{\mathrm{Vol}(\mathcal{B}_\beta)} \approx \pi \beta^2\frac{1}{K}\sum_{k=1}^K w_j^h(\xi_k) \phi(\xi_k) \eqqcolon \bar{p}_j(\phi)
\end{equation}
In the numerical approximation of the PDE solutions, denoted $w_j^h(\xi_k)$, we used the finite element method to solve a discretization of \eqref{eq:w-definition}. We used quadratic elements on a triangular mesh with a maximum element size of $h =.01$ using MATLAB PDE TOOLBOX\footnote{MATLAB Version (R2023a) and Partial Differential Equation Toolbox Version 3.10}.

With these probabilities in hand, we generated the vector of exit counts $\mathbf{N} = (N_1, N_2, \ldots N_J)$ through iid samples from the multinomial distribution weighted according to the exit probabilities conditioned on particle initial locations. That is to say, let $N = |\mathbf{N}|$ and $\{\mathcal{E}_n\}_{n=1}^N$ be iid with common distribution
\begin{displaymath}
    \mathcal{E} \sim \text{Multinomial}\big(p_0(\phi^0), p_1(\phi^0), \ldots p_J(\phi^0)\big)
\end{displaymath}
where $p_0(\phi^0)$ is the probability that a given particle does not exit through any of the detectors. Then for each detector $j$, we define $N_j = \sum_{n=1}^N \Ind_j(\mathcal{E}_n)$ and compute the vector of target probabilities 
\begin{displaymath}
\widehat{\mathbf{p}} = \mathbf{N}/N
\end{displaymath}
which we consider to be ``the data'' used for inference. When we wanted to remove variance due to individual path outcomes, we directly used the PDE quantities (the function $w$) and write $N = \infty$. 

To perform the gradient descent, we followed the algorithm implied by \Cref{defn:descent} to generate the sequence of locations $\{\vartheta_k\}_{k=1}^K$ with uniform step sizes $\alpha = 0.01$ and $K = 1000$ steps. In the numerical experiment supporting \Cref{fig:trajectories}, at each step we approximated $\widehat{\E^{\vartheta_k}}\big(\Ind_{X(\tau) \in D_j}\big)$ and $\widehat{\nabla_\theta \E^{\vartheta_k}}\big(\Ind_{X(\tau) \in D_j}\big)$ with $M = 10^4$ path samples. 

The start of the gradient descent for the source location was $(0.5,-0.05)$ in all cases. While the approach to the true source location parameter was successful in all four cases, it is interesting to note that the paths were distinct. In the absence of drift, the Brownian motion experiment and uniform scattering angle experiment yielded almost straightline searches. But the introduction of drift in the baseline Markov process resulted in curved trajectories.

\begin{figure}[h!]
    \centering
    \includegraphics[height = 2in,trim={0 0 1cm 0},clip]{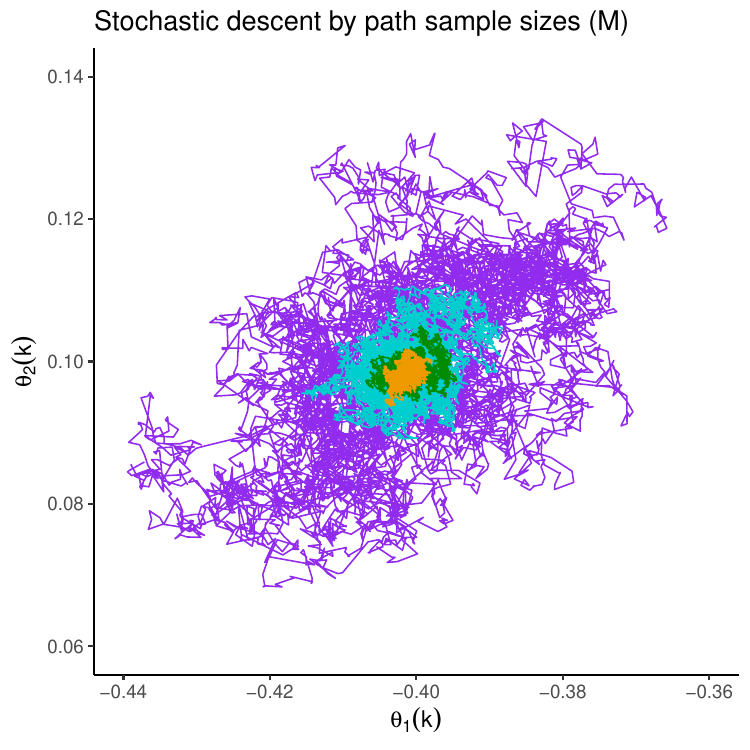}
    \includegraphics[height = 2.05in,trim={0cm 0 0 0},clip]{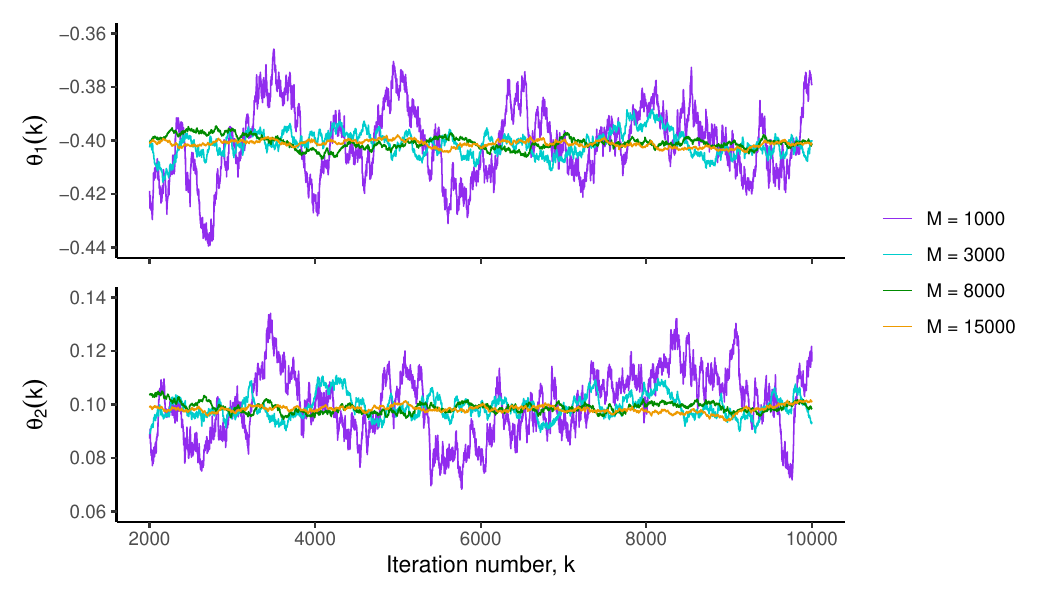}
    \caption{A further exploration of the stochastic gradient descent used in Experiment 1. After the initial approach to a neighborhood of the source parameter, we display the fluctuations that emerge as a function of the number of paths $M$ used for each estimation of the gradient. The iterations of the algorithm had the stepsize $h=0.01$ are shown after burn-in period of 2000 steps. In the experiment, there are five equally spaced detectors on the unit circle, with the source parameters $\theta^0 = (-0.4,0.1)$, $\beta^0 = 0.15$. The initial guess for the source location parameter was $\tilde\theta_0 = (0.5,-0.05)$. The target exit probabilities were obtained from numerical solutions of the appropriate $w_j$ PDE problems (i.e., $N = \infty$).} 
    \label{fig:search-M}
\end{figure}

\subsection{Assessing impact on simulated particle ensemble size.}
\label{sec:results:M}

To assess the impact of estimating probabilities and gradients through path simulations, we conducted an experiment for each ensemble size $M = \{1000, 3000, 8000, 15000\}$, following the form described in the previous section. The baseline Markov process for these experiments was Brownian motion with $\eta = 1/2$. The domain was the unit circle, and the true source location was $\theta^0 = (-0.4, 0.1)$ with 
$\beta^0 = 0.15$. We generated the data and the driving probability vector $\mathbf{p}$ in a manner similar to the previous section, with $N = \infty$. In this way, we minimized the error due to sample size as much as possible.

In \Cref{fig:search-M} we display estimator paths for the final 8000 values for $\tilde\theta_k$ in a single search for $\tilde\theta_{k,1}$ and $\tilde\theta_{k,2}$, respectively. As $M$ increases, the variance decreases as expected. We note that there is an apparent bias in the results, but this may be due to the finite sample size of the data. The exclusion of the true values is a reminder that what is displayed is not a proper construction of a confidence region. Such a quantity requires an understanding of variance due to data sample size which is not assessed here. 

\subsection{Numerical consistency for Brownian motion case}
\label{sec:results:N}

To demonstrate that the source identification problem is well posed for binned particle count data, we took a numerical approach. In contrast to the previous sections in which we looked to assess the validity of the gradient descent approximation, here we wish to establish that, as the number of observed particle exits increases, the estimator for the center of the source distribution improves. Following the terminology of the statistics literature, we say that an estimator is consistent if, viewing the estimator $\hat \theta_N$ as a deterministic function of the data when there are $N$ data points, $\lim_{N \to \infty} \Prob^{\theta_0}(|\hat \theta_N - \theta_0| > \epsilon) = 0$ for any $\epsilon > 0$. In the numerical experiments described below, we find that the rate of convergence matches the typical estimator rate $~N^{-1/2}$, \Cref{fig:N-consistency}. A rigorous study of consistency would require a study of the regularity of solutions to $u$ and $w$ that are beyond the scope of the present work.

In our numerical consistency experiments, we effectively taking $M=\infty$, i.~e.~assuming that the gradient descent finds the true global minimum when presented with an inference challenge. We approximate this by a brute force method, taking a collection of 10000 points near the true source parameter $\theta$ and for each $\theta$, computing an MC approximation for the exit probabilities of an ensemble of particles that have $\phi(\cdot \with \theta, \beta^0)$ for their initial location distribution. Adapting the notation presented in \eqref{eq:p-bar-mc}, $\bar p_j(\theta)$ denotes the probability of exit from detector $j$. The estimator can then be expressed
\begin{equation}
    \widehat{\theta}_\text{sweep}(\mathbf{p}) = \argmin_{\theta \in \domain} \left\{\sum_{j=1}^J \frac{1}{2} \Big|\bar{p_j}(\theta) - \widehat{p}_j\Big|^2\right\}.
\end{equation}

\begin{figure}
    \centering
    \begin{subfigure}{0.45\textwidth}
    \includegraphics[width = \textwidth]{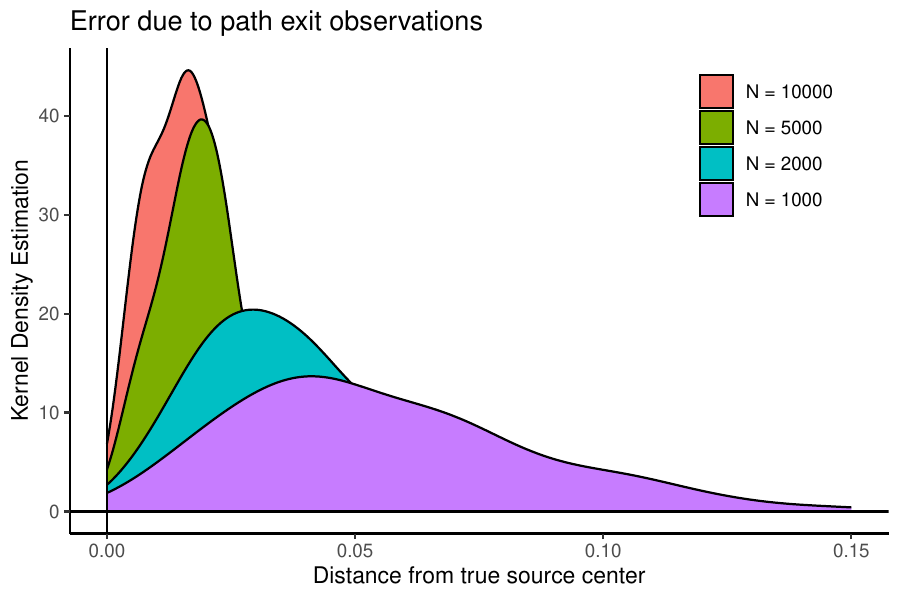}
    \caption{Distribution of estimator error due to finite data size.}
    \end{subfigure}
    \begin{subfigure}{0.45\textwidth}
    \includegraphics[width = \textwidth]{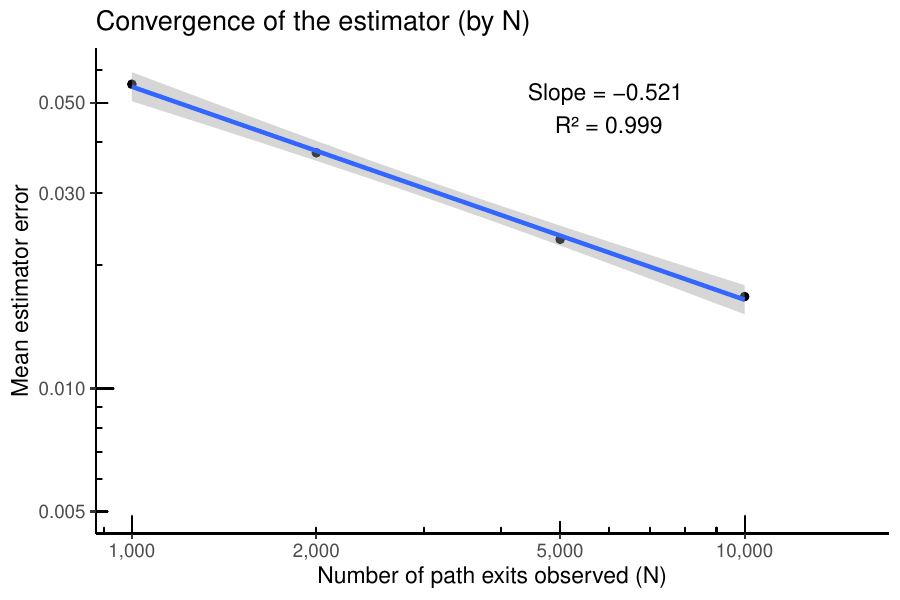}
    \caption{Mean error scaling $\sim N^{-1/2}$.\\
    \phantom{......}}
    \end{subfigure}

    \begin{subfigure}{0.45\textwidth}
    \includegraphics[width =\textwidth]{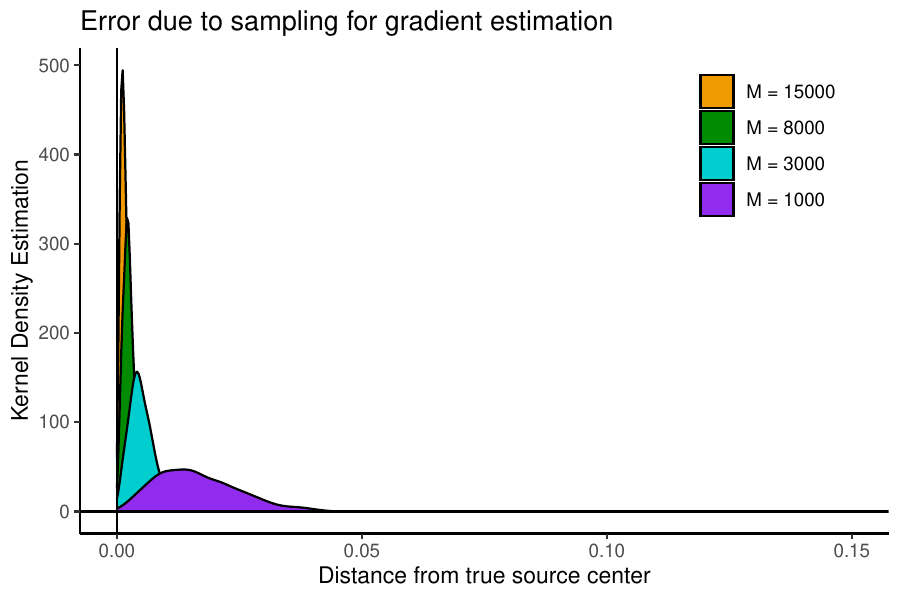}
    \caption{Distribution of estimator error due to gradient approximation by paths.}
    \end{subfigure}
    \begin{subfigure}{0.45\textwidth}
    \includegraphics[width = \textwidth]{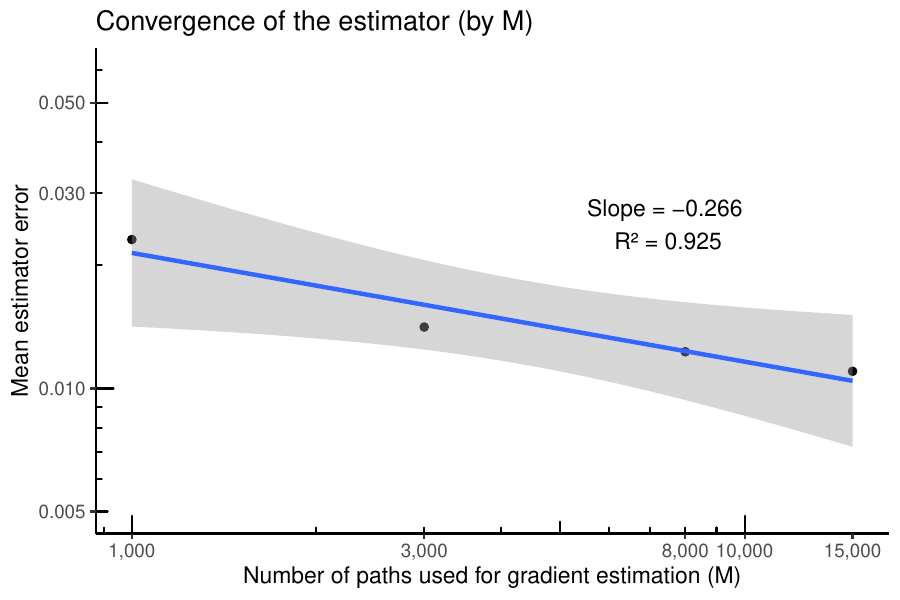}
    \caption{Mean error scaling $\sim M^{-1/4}$.\\
    \phantom{......}}
    \end{subfigure}
    \caption{Comparison of the studied estimator statistical properties in two cases: (i) a finite number of paths is observed to exit with the ``perfect'' stochastic gradient descent $(N<\infty,M=\infty)$ (top row), (ii) a ``perfect'' observation of the exit probabilities with limited estimation of gradients $(N=\infty,M<\infty)$ (bottom row).
     The observed scaling of the mean error with the sample size $\langle \text{|error|} \rangle \sim N^{-1/2}$ is as expected; however, the scaling (ii) is only $\langle|\text{error}|\rangle \sim M^{-1/4}$.}
    \label{fig:N-consistency}
\end{figure}

In the first panel of \Cref{fig:N-consistency}, we display a kernel density estimation for the distribution of distances from 
$\theta_{\text{sweep}}$ to the true source center $\theta^0 = (-0.4, 0.1)$ for $N \in \{1000, 2000, 5000, 10000\}$ The means of these distributions are displayed as a function of $N$ in the right panel, and the regression shows good agreement with an $N^{-1/2}$ scaling. An important initial note concerns what is meant by the ``number of observed particle exits'' though. In the simultaneous release system, we are able to observe particles that hit the detectors and those that do not. For the stochastic fountain, we only observe particles that hit the detectors. So, for the example we study here, there is approximately a $1/4$ probability that a given particle will hit a detector. One thousand samples in the fountain case is roughly equivalent to $4000$ samples in the simultaneous release system, assuming the true values of the particle birth rate $\lambda$ and the detector observation duration $T$ are known. 
Preliminary work on the consistency of the location scale parameter $\beta$ showed ambiguous results, and we reserve that inference challenge for future work.

\section{Discussion}

We have developed a method for finding the source of stochastically evolving particles based on their arrivals at a set of detectors on the boundary of a given domain. This is a stochastic version of classical source identification problems that appear in the literature for PDE inverse problems. In a sense, we have developed a ``stochastic-process-constrained optimization'' regime that is analogous to well-known ``PDE-constrained optimization'' methods. To this end, we have rigorously established the relationship between the boundary flux of steady-state PDEs with the exit rate of particles from a stochastic fountain ensemble that we define in \Cref{sec:fountain}. The stochastic version of the problem requires the computation of derivatives of mathematical expectations, which we tackle via a ``shape derivative'' method. Operating our investigation under the assumption that MC methods have advantages over deterministic numerical techniques in certain settings, we also demonstrated that gradients with respect to source distribution parameters can be computed pathwise with the same information necessary to produce exit probability results. 

Throughout this work, we emphasized finding the source location and assumed other parameters, particularly the source distribution's scale parameter $\beta$, are given. We believe that a targeted investigation is needed regarding estimation of $\beta$. In fact, for the Brownian motion example presented in \Cref{sec:results} the scale parameter can be unidentifiable (see \cite{elbadia2000inverse} for related discussion). If particle degradation is introduced, then the scale parameter may become identifiable, but with a much larger variance than source location estimator. Quantifying this question would follow from computation of the Fisher information of the estimators, but unlike the parameter gradients, we do not see at present how to estimate the Fisher information from path data alone.

The most significant limitation of the method we have presented is the need for simulated particles to regularly hit the boundary detectors. There exist adjoint methods that can improve efficiency of simulation \citep{smith2025stochastic}, but in complex geometries or situations where the detectors and source are small relative to $\domain$, substantial innovation may be required to improve performance. Nevertheless, the robustness of the method with respect to the underlying generating process is a promising note for this method.

\section{Acknowledgments}

We thank Darby Smith, Brian Franke, and Aaron Olson of Sandia National Labs, and Jonathan Mattingly of Duke University, for helpful comments and discussion.
RBL and PP thank Andreas Kyprianou and the MaThRad group for their hospitality and organization of a meeting to the Department of Statistics at the University of Warwick in which a preliminary version of this paper was presented. 

The work of RBL and SAM was supported by the Laboratory Directed Research and Development program (Project 222358) at Sandia National Laboratories, a multimission laboratory managed and operated by National Technology and Engineering Solutions of Sandia LLC, a wholly owned subsidiary of Honeywell International Inc. for the U.S. Department of Energy’s National Nuclear Security Administration under contract DE-NA0003525.

SAM was also supported by the NSF-Simons Southeast Center for Mathematics and Biology (SCMB) through NSF-DMS1764406 and Simons Foundation-SFARI 594594

\bibliographystyle{abbrvnat}
\bibliography{_source_id}

\end{document}